\documentclass[a4paper, 11pt]{amsart}
\usepackage[utf8]{inputenc}
\usepackage[T1]{fontenc}
\usepackage[english]{babel}
\usepackage{graphicx}
\usepackage{stmaryrd}
\usepackage{tikz}
\usepackage{tikz-cd} 
\usepackage[all]{xy}
\usepackage{comment}
\usepackage{bm}
\usepackage{comment}
\usepackage{amsmath,amsfonts,amssymb}
\usepackage{cite}
\usepackage[a4paper]{geometry}
\usepackage{amsthm}
\usepackage{float}
\usepackage{enumitem}
\usepackage{hyperref}
\usepackage{xcolor}
\usepackage{nicefrac}
\hypersetup{
	colorlinks=true,
	linkcolor={blue},
	citecolor={blue},
	urlcolor={blue}}
	
	\usepackage{cleveref}  % Ismael: I added this to use \Cref
	
\newtheorem{thm}{Theorem}[section]
\newtheorem{prop}[thm]{Proposition}

\newtheorem{cor}[thm]{Corollary}
\newtheorem{conj}[thm]{Conjecture}
\newtheorem{claim}[thm]{Claim}
\newtheorem{qu}[thm]{Question}
\newtheorem{lemma}[thm]{Lemma}
\newtheorem{lem}[thm]{Lemma}
\theoremstyle{definition}
\newtheorem{notation}[thm]{Notation}
\newtheorem{defn}[thm]{Definition}

\theoremstyle{remark}
\newtheorem{rmk}[thm]{Remark}
\newlength{\plarg}
\usetikzlibrary{cd}
\calclayout

\newcommand{\sub}{\subseteq}
\newcommand{\vb}{\mathrm{vb}}
\newcommand{\cd}{\mathrm{cd}}

\newcommand{\G}{\mathcal G}
\renewcommand{\AA}{{\bf A}}
\newcommand{\BB}{{\bf B}}
\newcommand{\CC}{{\mathcal C}}

\newcommand{\GG}{{\bf G}}
\newcommand{\HH}{{\bf H}}
\newcommand{\NN}{{\bf N}}
\newcommand{\UU}{{\bf U}}
\newcommand{\KK}{{\bf K}}

\newcommand{\ZZ}{{\bf Z}}
\renewcommand{\b}{b_1^{(2)}}
\renewcommand{\d}{d_1^{(2)}}
\newcommand{\vd}{{\mathrm vd}_1^{(2)}}
\newcommand{\ti}{\widetilde}
\newcommand{\F}{\mathbb{F}}
\newcommand{\Z}{\mathbb{Z}}
\newcommand{\Q}{\mathbb{Q}}
\renewcommand{\hat}{\widehat}

\newcommand{\ga}{\gamma}
\newcommand{\Ga}{\Gamma}

\newcommand{\hp}{\hat{p}}
\newcommand{\lan}{\langle}
\newcommand{\ran}{\rangle}
\newcommand{\lrar}{\longrightarrow}
\newcommand{\rar}{\rightarrow}
\newcommand{\n}{\unlhd}
\newcommand{\nc}{\unlhd_{c}}
\newcommand{\no}{\unlhd_{o}}
\newcommand{\lo}{\leq_{o}}
\newcommand{\lc}{\leq_{c}}
\newcommand{\FP}{\mathtt{FP}}
\newcommand{\po}{\colon}
\renewcommand{\bar}{\overline}
\newcommand{\RG}{\mathrm{RG}}
\newcommand{\Aut}{\mathrm{Aut}}

\newcommand{\D}{\mathcal D}

\theoremstyle{plain}
\newtheorem{mainthm}{Theorem}

\newtheorem{maincor}[mainthm]{Corollary}

\begin{document}
	
	\title{Virtual homology of limit groups and profinite rigidity of direct products}
	\author{Jonathan Fruchter and Ismael Morales}
	\smallskip
	
	\begin{minipage}{\linewidth}

		\vspace{0mm}
		
		\begin{abstract} We show that the virtual second Betti number of a finitely generated, residually free group $G$ is finite if and only if $G$ is either free, free abelian or the fundamental group of a closed surface. We also prove a similar statement in higher dimensions. We then develop techniques involving rank gradients of pro-$p$ groups, which allow us to recognise direct product decompositions. Combining these ideas, we show that direct products of free and surface groups are profinitely rigid among finitely presented, residually free groups, partially resolving a conjecture of Bridson's. Other corollaries that we obtain include a confirmation of Mel'nikov's surface group conjecture in the residually free case, and a description of closed aspherical manifolds of dimension at least $5$  with a residually free fundamental group.
	
        \end{abstract}
        
	\end{minipage}
	%	The virtual $n$-th Betti number of a finitely generated group $G$ is defined as $\mathrm{vb}_n(G)=\sup \{\dim_{\mathbb{Q}} H_n(H;\mathbb{Q})\, \vert \, H \text{ is a finite-index subgroup of } $G$\}$. We calculate the virtual Betti numbers of several families of groups. For instance, if $L$ is a non-abelian limit group and $3\leq n\leq \cd(L)$, then $\mathrm{vb}_n(L)=\infty$. The virtual second Betti number is of particular interest. We show that the virtual second Betti number of a finitely generated, residually free group $G$ is finite if and only if $G$ is either free, free abelian or a surface group (i.e. the fundamental group of a closed and connected surface of non-positive Euler characteristic). To do so, we show that if a limit group $L$ is not free, free abelian or a surface group, then it contains a subgroup isomorphic to $\pi_1(\Sigma) \ast \mathbb{Z}$, where $\pi_1(\Sigma)$ is a surface group. This relies on Henry Wilton's proof that non-free limit groups contain a surface subgroup. \par \smallskip We apply these results to the study of profinite rigidity of direct products of groups, and show that direct products of free and surface groups are profinitely rigid among finitely presented, residually free groups, partially resolving a conjecture of Bridson. Other corollaries that we obtain include a confirmation of Mel'nikov's surface group conjecture in the residually free case and a description of closed, aspherical manifolds of dimension at least $5$  with a residually free fundamental group.
	\thispagestyle{empty}

 \maketitle

	\vspace{5mm}
	
	\tableofcontents
	
	\vspace{5mm}
	
\section{Introduction}
The main theme of this paper is to utilize the homology of finite-index subgroups of a given group, in order to recognize whether it is a free or a surface group. We do so within two classes of groups: (finitely generated) \textbf{residually free groups} and word-hyperbolic fundamental groups of graphs of free groups amalgamated along cyclic subgroups. For abbreviation and readability, we refer to the latter as \textbf{HGFC-groups} (which stands for \textbf{H}yperbolic \textbf{G}raphs of \textbf{F}ree groups with \textbf{C}yclic edge groups). Throughout the paper, we assume that the edge maps of any graph of groups decomposition are injective, and that cyclic edge groups are always infinite cyclic (unless stated otherwise). \par \smallskip

Given a finitely generated group $G$ and a field $k$ (endowed with a trivial $G$-action), we denote the \emph{$i$-th Betti number} of $G$ with coefficients in $k$ by $b_i^k(G)=\dim_k H_i(G; k)$. We focus on the \emph{virtual $i$-th Betti number} of $G$, which allows us to study $G$ by looking at the homology of all of its finite-index subgroups at once.

\begin{defn} \label{virtual_b}
    Let $G$ be a finitely generated group and let $k$ be a field. The \emph{virtual $i$-th Betti number} of $G$ with coefficients in $k$ is given by
    \begin{equation*}
        \mathrm{vb}_n^k(G)=\sup \{\dim_{k} H_i(H;k)\, \vert \, H \text{ is a finite-index subgroup of } G\}
    \end{equation*}
\end{defn}

\begin{rmk}
    When $k=\Q$, we simply write $b_i(G)$ and $\vb_i(G)$ to denote $b_i^{\Q}(G)$ and $\vb_i^{\Q}(G)$ respectively.
\end{rmk}

One of many remarkable properties of finitely generated residually free groups, is that many group-theoretic properties can be read from their (virtual) homology; for example, the failure of a finitely generated residually free group $G$ to satisfy certain finiteness properties, can be detected in the homology of finite-index subgroups of $G$ \cite[Theorem B]{Bri09}. This makes residually free groups a compelling object to study via virtual homology.

\subsection{Virtual homology of residually free groups} Wilton showed in \cite{surfacesubs} that every one-ended HGFC-group contains a surface subgroup, and deduced that its virtual second Betti number, with coefficients in any field $k$, is positive. We extend this result and give a full classification of HGFC-groups by their virtual second Betti number. This classification is remarkably simple: apart from the obvious cases, the virtual second Betti number of an HGFC-group is infinite. \par \smallskip

HGFC-groups also play a special role in the study of \emph{limit groups}, that is, finitely generated, \emph{fully residually free} groups (for the full definition and a few key properties of limit groups see Subsection \ref{limintro}). This class of groups has been extensively studied over the past sixty years, and gained popularity due to its importance in Sela's proof that all non-abelian free groups share the same first-order theory, answering a longstanding question of Tarski's (see \cite{sela1} \emph{et seq.} as well as \cite{Kharlampovich1998} \emph{et seq.}). Limit groups exhibit a \emph{cyclic hierarchical structure}: they can be built up from easy-to-understand building blocks (namely free and free abelian groups), taking cyclic amalgams and HNN extensions repeatedly. In particular, HGFC-groups lie (almost) at the bottom of this hierarchical structure of a limit group. We also derive a classification for limit groups, and more generally residually free groups, in terms of their virtual second Betti number.

%In conjunction with the aforementioned classification of hyperbolic fundamental groups of graphs of free groups with cyclic edge groups by their virtual second Betti number, we derive a similar classification for limit groups. Baumslag showed that a residually free group is a limit group if and only if it is finitely generated, and it does not contain $F \times \mathbb{Z}$ as a subgroup, where $F$ is a free group of rank $2$ \cite{Bau67a}. We therefore deduce:

\begin{mainthm}[Theorem \ref{mainlim2}]
	\label{mainlim}
	 Let $G$ be a finitely generated residually free  group or an HGFC-group. Then
    \begin{enumerate}
			\item $\mathrm{vb}_2^k(G)=0$ if and only if $G$ is free,
			\item $\mathrm{vb}_2^k(G)=1$ if and only if $G \cong \pi_1(\Sigma)$ where $\Sigma$ is a closed, connected surface,
            \item $\mathrm{vb}_2^k(G)=\binom{d}{2}$ if and only if $L\cong \mathbb{Z}^d$ (for $d>2$)
			\item $\mathrm{vb}_2^k(G)=\infty$ otherwise.
		\end{enumerate}

    Furthermore, if $G$ is a limit group, then for any $n\ge 3$ we have that $\mathrm{vb}_n^k(L)=\infty$ unless one of the following two holds:
    \begin{enumerate}
        \item $\mathrm{cd}(L)<n$, in which case $\mathrm{vb}_n^k(L)=0$, or
        \item $L$ is free abelian of rank at least $n$, in which case $\mathrm{vb}_n^k(L)=\binom{\mathrm{rank}(L)}{2}$.
    \end{enumerate}
\end{mainthm}
Therefore, surface groups are characterized among limit groups by their virtual second Betti number. We now explain the idea of the proof of Theorem \ref{mainlim}. Given such a group $G$, which does not fit into one of the exceptional cases, we consider two possibilities. Either $G$ is not a limit group, and hence it contains a subgroup isomorphic to $F_2\times \Z$; or it is a limit group which contains a subgroup of  the form $\pi_1(\Sigma) \ast \Z$, where $\Sigma$ is a closed, connected surface. Since both $F_2\times \Z$ and $ \pi_1(\Sigma) \ast \Z$ have infinite virtual second Betti number,  we obtain that $\mathrm{vb}_2^k(G)=\infty$ utilizing virtual retractions.

\Cref{mainlim} also allows us to describe the structure of closed, aspherical manifolds with a residually free fundamental group. We suspect that experts are probably familiar with the description below (see Remark \ref{mflds_rmk}), but we record it as a consequence of Theorem \ref{mainlim}.

\begin{maincor}[Corollary \ref{mflds2}] \label{mflds}
    Let $M$ be an aspherical, closed manifold of dimension $n$. Then:
    \begin{enumerate}
        \item if $\pi_1(M)$ is fully residually free and $n\ge 3$ then $M \cong T^n$.
        \item if $\pi_1(M)$ is residually free and $n\ge 5$ then $M$ has a finite cover that is homeomorphic to the direct product of a torus and finitely many closed surfaces.
    \end{enumerate}
\end{maincor}

Virtual Betti numbers of finitely generated groups have been studied before on numerous occasions (see, for instance, \cite{Bridson2015, dessi, agol, Venkataramana2008-nl, Cooper2007}). In fact, Agol's resolution of the Virtual Haken conjecture \cite{Agol13} shows that the first and second virtual betti numbers of closed hyperbolic three manifolds are infinite (see also Remark \ref{agol_reid_rmk}). In  \Cref{virtual_Bettis} we give more examples of calculations of virtual Betti numbers. \par \smallskip

%Virtual Betti numbers of finitely generated groups have been studied before on numerous occasions: in \cite{Bridson2015}, Bridson and Kochloukova show that the virtual first Betti number of a finitely presented nilpotent-by-abelian-by-finite group is finite, and in \cite{dessi}, Kochloukova and Mokari show that the same holds for some abelian-by-polycyclic groups. Agol \cite{agol}, Venkataramana \cite{Venkataramana2008-nl} and Cooper, Long and Reid \cite{Cooper2007} each show that under some conditions the virtual first Betti number of arithmetic hyperbolic $3$-manifold groups is infinite. 

Another work related to this paper is \cite{l2_lim}, in which Bridson and Kochloukova show that the second \emph{$L^2$-Betti number} of a limit group $L$ is always $0$. By L\"uck's approximation theorem \cite{Luc94}, the second $L^2$-Betti number of a limit group $L$   is equal to $\lim_{n \rightarrow \infty} \frac{\dim H_2(H_n;\mathbb{Q})}{[L:H_n]}$ where $(H_n)$ is any nested sequence of normal subgroups of finite-index in $L$ with trivial intersection, i.e. $\bigcap_n H_n = \{1\}$. By the proof of Theorem \ref{mainlim}, one can easily produce a nested sequence $(G_n)$ of finite-index normal subgroups of a given limit group $L$, with $\lim_{n\rightarrow \infty} [L:G_n]=\infty$, satisfying 
\[\liminf_{n \rightarrow \infty} \frac{\dim H_2(G_n;\mathbb{Q})}{[L:G_n]}>0.\] However, for the reasons above, such a sequence $(G_n)$ can never have a trivial intersection. 
	\par 
	\smallskip

As a last note on $L^2$-Betti numbers, we can also establish, partly using \Cref{mainlim}, a similar dichotomy using other virtual homological invariants.

\begin{mainthm}[\Cref{boundedvd}] \label{boundedvdintro} Let $G$ be a non-abelian limit group that is not isomorphic to either a free or a surface group. Then for all $C>0$ there exists a finite-index subgroup $H\leq G$ such that $b_1(H)-\b(H)\geq C$. 
\end{mainthm}
An interesting aspect of \Cref{boundedvdintro} is its relation with profinite invariants. Lubotzky proved in \cite[Proposition 1.4]{Lub14} that the property of being residually-$p$ is not a profinite invariant.  On the positive side, Jaikin-Zapirain proved in \cite[Theorem 1.1] {And212} that if $G$ is a group in the finite or solvable genus of a finitely generated free or surface group $S$, then it is residually-$p$ for all primes $p$. The proof relies, crucially, on the fact that free and surface groups are the exceptional cases of \Cref{boundedvdintro}. This illustrates that the methods of \cite[Theorem 1.1] {And212} do not extend directly to prove that, given a limit group $L$, any group  in the solvable genus  of $L$ is residually-$p$. 
\subsection{The surface group conjectures}

We recall a duo of famous conjectures revolving around surface groups, inspired by a question of Mel'nikov's. A group $G$ is called a \emph{Mel'nikov group} if it is a non-free infinite one-relator group, all of whose finite-index subgroups are one-relator groups (see \cite[Definition 1.1 and following discussion]{Kie22}). The Surface Group Conjectures, which appear in numerous papers (as they appear in Baumslag--Fine--Rosenberger's book \cite[Conjecture 2.17 and Question 2.18]{melnikov}) claim the following:

\begin{conj}[Surface Group Conjecture A]
Let $G$ be a residually finite Mel'nikov group. Then $G$ is either a surface group or $G\cong \mathrm{BS}(1,n)$ (for some $n \ne 0$).
\end{conj}

\begin{conj}[Surface Group Conjecture B]
Let $G$ be a Mel'nikov group, all of whose infinite-index subgroups are free. Then $G$ is a surface group.
\end{conj}

%It is worth noting that Fine, Kharlampovich, Myasnikov, Remeslennikov, and Rosenberger posed a third Surface Group Conjecture \cite[Surface Group Conjecture C]{russian_surface}, which states that a non-free and freely indecomposable limit group, all of whose infinite-index subgroups are free, must be a surface group. The Surface Group Conjecture C was resolved by Wilton \cite[Corollary 5]{wilton-one-ended}, and it follows from \cite[Theorem 3]{wilton-one-ended} which states that every one-ended HGFC-group which is not a surface group contains a one-ended infinite-index subgroup. It was also independently confirmed by Ciobanu, Fine and Rosenberger in \cite[Theorem 3.2]{ciobanu}, where they also show that the Surface Group Conjecture A holds for certain one-relator groups. 
Wilton proved Conjecture B in the case where $G$ is residually free in \cite{wilton-one-ended}. Since one-relator groups have second Betti number at most one \cite{Lyndon50}, Theorem \ref{mainlim} also confirms Surface Group Conjecture A in the same context.
\begin{maincor} \label{Melni} Every residually free Mel'nikov group is a surface group.
\end{maincor}
Recently,  Gardam, Kielak and Logan established that both Surface Group Conjectures A and B are true for two-generated one-relator groups \cite[Theorem 1.5]{Kie22}. In addition, Jaikin-Zapirain and the second author \cite{JaiMor23} proved Surface Group Conjecture A in the case where $H_2(G; \Z)\neq 0$.

\subsection{Profinite rigidity within the class of limit groups}

From a logician's point of view, limit groups are notoriously hard to distinguish from one another: they all have the same universal first-order theory. This makes the question of telling limit groups apart from one another a compelling one. A key ingredient in studying profinite invariants of limit groups is \emph{cohomological goodness}, an important notion introduced by Serre in \cite{galois_cohomology}. Recall that a group $G$ is \emph{cohomologically good} (or just \emph{good}) if for every finite $G$-module $M$ the map $H^n(\hat{G};M)\rightarrow H^n(G;M)$ on cohomology groups, induced by the canonical map $G\rightarrow \hat{G}$, is an isomorphism. Grunewald, Jaikin-Zapirain and Zalesskii showed that limit groups are good \cite[Theorem 1.3]{limgood}; the same holds for HGFC-groups \cite{hag-wise, graphgood}. Wilton used this to show that free and surface groups are profinitely rigid among limit groups in \cite[Corollary D]{surfacesubs} and \cite[Theorem 2]{profinite_surface}, respectively. Theorem \ref{mainlim} gives an alternative proof of the latter theorem.

\begin{maincor}
		\label{rigidityintro}
		Let $\pi_1(\Sigma)$ be the fundamental group of a closed, connected surface and let $G$ be a limit group or a  hyperbolic group that splits as a graph of free groups with cyclic edge groups. If $\hat{G}=\widehat{\pi_1(\Sigma)}$ then $G\cong \pi_1(\Sigma)$.
	\end{maincor}
 
%\begin{proof}
 %   Since $\hat{G}\cong \widehat{\pi_1(\Sigma)}$ and $G$ and $\pi_1(\Sigma)$ are both cohomologically good,
    %\begin{equation*}
     %   \sup\{\dim_{\mathbb{F}_p} (H^2(H;\mathbb{F}_p)) \vert H\le G \text{ of finite index}\}=1.
    %\end{equation*}
    %Theorem \ref{mainlim} and Remark \ref{cohomological_version} imply that $G$ is a surface group. Surface groups are distinguished from one another by their abelianizations, and therefore by their finite quotients. Hence $G \cong \pi_1(\Sigma)$.
%\end{proof}

%Due to the nature of the classification of limit groups and HGFC-groups by their virtual Betti numbers, one cannot deduce further profinite rigidity results by looking at the virtual homology of limit groups (apart from the obvious cases, when $G$ is free or free abelian). In Section \ref{questions} we discuss more refined virtual homological invariants of groups which might aid in profinite recognition of more complicated limit groups.

\subsection{Profinite rigidity and direct products}

A celebrated result of Platonov and Tavgen \cite{platonov} states that a direct product of two non-abelian free groups is not profinitely rigid: given two non-abelian free groups $F$ and $F'$, there exist (infinitely many non-isomorphic) finitely generated subgroups $H\le F \times F'$ of infinite index such that the inclusion $H \hookrightarrow F \times F'$ induces an isomorphism of profinite completions. Since a finitely presented subgroup of $F \times F'$ that intersects both factors non-trivially and maps onto each factor must be of finite index \cite{baumslag_product}, there are no such examples $H\le F \times F'$ as above where $H$ is finitely presented. This raises the question of whether direct products of free groups are profinitely rigid among finitely presented, residually free groups. We answer this to the positive, with the following stronger result:

\begin{mainthm}[Theorem \ref{products2}] \label{productsintro} Let $G$ be a finitely presented residually free group. Let $S_1, \dots, S_n$ be free or surface groups and let $\Ga=S_1\times  \cdots \times S_n$ be their direct product. If $\hat G\cong \hat \Ga$, then $G\cong \Ga$.
\end{mainthm}

This resolves a particular case of Bridson's conjecture \cite[Conjecture 7]{bridson} that direct products of free groups are profinitely rigid among finitely presented, residually finite groups. It is natural to ask if there is a weaker finiteness property that suffices to ensure the conclusion of \Cref{productsintro}. The property of being $\FP_2(\Q)$ is, in general, weaker than finite presentability. However, in the context of residually free groups, they are equivalent \cite[Theorem D]{BridsonHowie13}. For the proof of \Cref{productsintro} (\Cref{products2}), we develop in \Cref{ProductsSection}  some techniques on rank gradients of pro-$p$ groups (a pro-$p$ analogue of $L^2$-Betti numbers) which will help us in recognizing direct product decompositions; we believe that these techniques may be of independent interest (see, for instance, Propositions \ref{s1} and \ref{s2}). \par \smallskip
%We conclude the introduction by mentioning that, while 
While it remains unknown whether direct products of free and surface groups are profinitely rigid among finitely presented groups (as recalled in the previous paragraph), we note that Bessa, Grunewald and Zalesskii gave examples of non-isomorphic groups $G$ and $H$ that are virtually direct products of free and surface groups (and hence finitely presented), satisfying $\hat{G} \cong \hat{H}$ \cite[Corollary 3.10]{Bessa2014}. Bauer found similar examples using different methods in \cite[Section 7]{Bauer2014}.% where he constructs pairs of Beauville surfaces with non-isomorphic fundamental groups that have isomorphic profinite completions.

\subsection{Acknowledgements} 
The first author was supported by a Clarendon Scholarship and a John Moussouris Scholarship at the University of Oxford.  The second author is supported by the Oxford--Cocker Graduate Scholarship.  The authors would like to thank Henry Wilton for an interesting conversation and Ashot Minasyan and Alan Reid for insightful comments. Finally, special thanks are dedicated to Martin Bridson, Dawid Kielak and Richard Wade for their advice, suggestions and generous support.

\section{Preliminaries} \label{PrelimSection}

\subsection{Graphs of spaces}
Given an HGFC-group, we will often realize its graph of groups decomposition as a \emph{graph of spaces}. We therefore begin by defining graphs of spaces and recollecting a few basic facts. For a detailed account of the theory of graphs of spaces we refer the reader to Scott and Wall \cite{scottwall}.
	\begin{defn}
		A \emph{graph of spaces} $X$ consists of the following data:
		\begin{itemize}
			\item a graph $\Xi$,
			\item for each vertex $v\in \mathrm{V}(\Xi)$ a connected CW-complex $X_v$,
			\item for each edge $e \in \mathrm{E}(\Xi)$ a connected CW-complex $X_e$ and two $\pi_1$-injective maps $\partial^\pm_e:X_e\rightarrow X_{v^\pm}$ (called \emph{attaching maps}) where $v^+$ and $v^-$ are the endpoints of the edge $e$.
		\end{itemize}
	\end{defn}
	Note that each graph of spaces $X$ has a topological space naturally associated to it, namely its \emph{geometric realization}. The geometric realization of $X$ is defined as
	\begin{equation*}
		\Bigg( \bigsqcup_{v \in \mathrm{V}(\Xi)} X_v \sqcup \bigsqcup_{e \in \mathrm{E}(\Xi)} X_e\times[-1,1] \Bigg) \Bigg/ \sim
	\end{equation*}
	where the equivalence relation $\sim$ identifies $(x,\pm 1)\in X_e\times [-1,1]$ with $\partial_e^\pm(x)\in X_{v^\pm}$. We will often abuse notation and use $X$ to refer to the geometric realization of $X$. \par 
	\smallskip 
	By the Seifert–Van Kampen theorem, the fundamental group of $X$ admits a graph of groups decomposition: $\pi_1(X)$ is the fundamental group of the graph of groups with underlying graph $\Xi$ and whose vertex and edge groups are $\{\pi_1(X_v) \vert v\in \mathrm{V}(\Xi)\}$ and $\{\pi_1(X_e) \vert e\in \mathrm{e}(\Xi)\}$ respectively. The edge maps of this graph of groups are given, up to conjugation, by $(\partial_e^\pm)_*:\pi_1(X_e)\rightarrow \pi_1(X_{v^\pm})$. We will denote this graph of groups decomposition by $\mathcal{G}(X)$ and refer to the vertex and edge groups of $\mathcal{G}(X)$ as $G_v$ and $G_e$. \par 
	\smallskip
	We remark that from this perspective, an HGFC-group can always be viewed as the fundamental group of a graph of spaces whose vertex spaces are graphs and whose edge spaces are all isomorphic to $S^1$. \par 
	\smallskip
 One situation where utilizing the machinery described above is particularly useful, is studying finite-sheeted coverings of graphs of spaces; we will make use of that in the proof of Lemma \ref{cyclic}. The vertex and edge spaces of a finite-sheeted cover $X'$ of a graph of spaces $X$ are finite-sheeted covers of the edge and vertex spaces of $X$. The attaching maps of $X'$ are called \emph{elevations}, and the elevations of an attaching map ${\partial_e}^\pm:X_e\rightarrow X_{v^\pm}$ of $X$ can be defined as the maps appearing at the top of the following pullback diagram
    \begin{equation*}
        \xymatrix{
        \bigsqcup_{e'}X_{e'} \ar[d] \ar[rr]^{\bigsqcup_{e'}{\partial_{e'}^\pm}}  & & X_{{v'}^\pm} \ar[d]\\
        X_e \ar[rr]^{{\partial_e}^{\pm}} & & X_v
        }            
    \end{equation*}
    where $X_{{v'}^\pm}$ are covering spaces of $X_{v^\pm}$ and $\bigsqcup_{e'}X_{e'}$ is the disjoint union of the edge spaces of $X'$ that lie above $X_e$. We remark that under certain conditions, a collection of covering spaces of edge and vertex spaces of $X$ can be formed to create a new graph of spaces $Y$ called a \emph{precover} of $X$; the key feature of precovers, is that their fundamental groups are in correspondence with (infinite index) subgroups of $\pi_1 X$. Elevations of attaching maps of $X$ to $Y$, that are not attaching maps of $Y$, are called \emph{hanging elevations} of $Y$. We refer the reader to \cite[Section 1]{wilton-one-ended} for a detailed account of the above. \par \smallskip

We finish this subsection by proving a technical lemma which will allow us to prove \Cref{boundedvd}. The role it plays is that of replicating usual or $L^2$-homology classes in finite covers. %and will also be used in Section \ref{virtual_Bettis}.
	\begin{lemma}
		\label{cyclic}
		Suppose that a finitely generated and subgroup separable group $G$ splits as a finite and connected graph of spaces $X$ whose vertex groups are not virtually cyclic. Suppose furthermore that there is a vertex group $G_v$ of $\mathcal{G}(X)$ with $\mathrm{b}_n^k(G_v)>0$ for some $n\in \mathbb{N}$ and a field $k$. If every edge group $G_e$ of $\mathcal{G}(X)$ satisfies $b_n^k(G_e)=0$, and at least one of the following holds,
		\begin{enumerate}
			\item the underlying graph $\Xi$ of $X$ is not a tree,
			\item all of the edge groups of $\mathcal{G}(X)$ are cyclic,
		\end{enumerate}
		then for every $\ell \in \mathbb{N}$ there is a finite cover $\hat{X}$ of $X$ with $\mathrm{b}_n^k(\pi_1(\hat{X}))\ge \ell$. In particular, $\mathrm{vb}_n^k(G)=\infty$.
	\end{lemma}
	
	\begin{proof}

		Suppose first that $\Xi$ is not a tree, so there is an edge $e\in \mathrm{E}(\Xi)$ with $\Xi - \{e\}$ a connected graph. Take $\ell$ copies of $X$ and enumerate them $X_1$,\ldots,$X_\ell$; let $e_i$ be the copy of $e$ in $\mathrm{E}(\Xi_i)$. Remove $X_{e_i}\times(-1,1)$ from $X_i$ to obtain a precover $X'_i$ of $X_i$ with two hanging elevations $\partial_i^\pm$. Let $\hat{X}$ be the space obtained by splicing together $X'_1,\ldots,X'_\ell$, pairing the elevation $\partial_i^+$ with $\partial_{i+1}^-$ (and $\partial_\ell^+$ with $\partial_1^-$). The resulting space for $\ell=3$ appears in Figure \ref{fig_circ}.
		\par 
		\smallskip
		Note that $\hat{X}$ contains $\ell$ copies of $X_v$ as vertex spaces. Since $\mathrm{b}_n^k(G_v)>0$ and for every edge $e\in \mathrm{E}(\Xi)$ we have that $\mathrm{b}_n^k(G_e)=0$, a repeated use of Mayer-Vietoris shows that $b_n^k(\hat{G}) \ge \sum _{\hat{v} \in \mathrm{V}(\hat{\Xi})} b_n^k(\hat{G}_{\hat{v}})) \ge \ell$.  \par 
		\smallskip

  \begin{figure}
			\centering
			\begin{tikzpicture}
				
				\usetikzlibrary{patterns}
				\node at (12,1.5) [rectangle,draw] (X) {$X_u$};
				\node at (12,3.5) [rectangle,draw] (Y) {$X_v$};
				\node at (10.5,3.5) [rectangle, draw] (Z) {$X_w$};
				\draw    (X) to[out=70,in=-70] node [pos=0.5, draw, fill=green, shape=circle, minimum size=5.0pt, inner sep=0pt] {} (Y);
				\draw[red,dashed]    (X) to[out=110,in=-110] node [pos=0.5, draw=black, fill=orange, solid, shape=circle, minimum size=5.0pt, inner sep=0pt] {} (Y);
				\draw[blue] (Z) to node[midway, draw=black, fill=teal, shape=circle, minimum size=5.0pt, inner sep=0pt] {} (Y);
				
				\path [draw,->>, thick] (8,2.5) -- node[above] {$p$} (11,2.5);
				
				\node at (3,0) [rectangle,draw] (Xv2) {$X_u$};
				\node at (3,2) [rectangle,draw] (Y2) {$X_v$};
				\node at (1.5,2) [rectangle, draw] (Z2) {$X_w$};

				\node at (7,0) [rectangle,draw] (Xv3) {$X_u$};
				\node at (7,2) [rectangle,draw] (Y3) {$X_v$};
				\node at (5.5,2) [rectangle, draw] (Z3) {$X_w$};

				\node at (5,3) [rectangle,draw] (Xv4) {$X_u$};
				\node at (5,5) [rectangle,draw] (Y4) {$X_v$};
				\node at (3.5,5) [rectangle, draw] (Z4) {$X_w$};
				
				\draw    (Xv2) to[out=70,in=-70] node [pos=0.5, draw, fill=green, shape=circle, minimum size=5.0pt, inner sep=0pt] {} (Y2);
				\draw    (Xv3) to[out=70,in=-70] node [pos=0.5, draw, fill=green, shape=circle, minimum size=5.0pt, inner sep=0pt] {} (Y3);
				\draw    (Xv4) to[out=70,in=-70] node [pos=0.5, draw, fill=green, shape=circle, minimum size=5.0pt, inner sep=0pt] {} (Y4);
				
				\draw[red]    (Xv2) to[out=45,in=-90] node [pos=0.5, draw=black, fill=orange, shape=circle, minimum size=5.0pt, inner sep=0pt] {} (Y3);
				\draw[red]    (Xv3) to[out=45,in=-45] node [pos=0.6, draw=black, fill=orange, shape=circle, minimum size=5.0pt, inner sep=0pt] {} (Y4);
				\draw[red]    (Xv4) to[out=100,in=100] node [pos=0.5, draw=black, fill=orange, shape=circle, minimum size=5.0pt, inner sep=0pt] {} (Y2);
				\draw[blue] (Z2) to node[midway, draw=black, fill=teal, shape=circle, minimum size=5.0pt, inner sep=0pt] {} (Y2);
				\draw[blue] (Z3) to node[midway, draw=black, fill=teal, shape=circle, minimum size=5.0pt, inner sep=0pt] {} (Y3);
				\draw[blue] (Z4) to node[midway, draw=black, fill=teal, shape=circle, minimum size=5.0pt, inner sep=0pt] {} (Y4);

			\end{tikzpicture}    \caption{Splicing together three copies of a precover of $X$ along the red edge}
			\label{fig_circ}
		\end{figure}
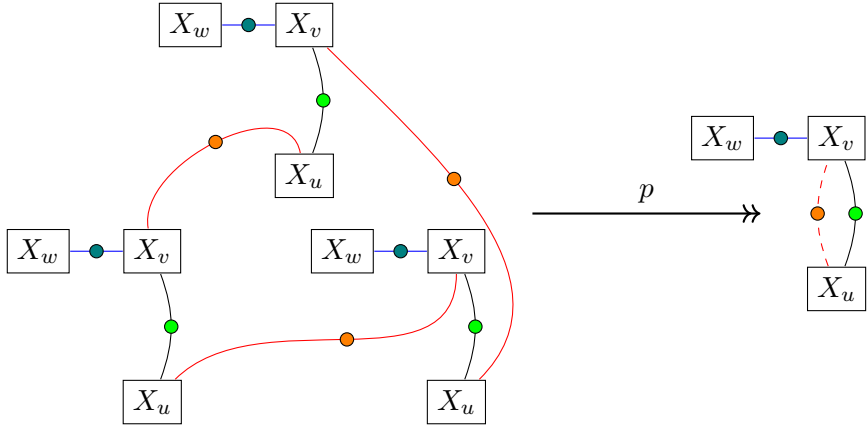
		We still need to treat the case where $\Xi$ is a tree; we further assume that all of the edge groups of $\mathcal{G}(X)$ are cyclic. Let $X_u$ be a vertex space that is adjacent to $X_v$, and let $e\in \mathrm{E}(\Xi)$ be the edge connecting $v$ and $u$. Since $G_u$ is not cyclic, there exists $g\in G_u$ that lies outside of $G_e \le G_u$. By assumption, the cyclic subgroup $G_e$ is separable in $G_u$, so there is a finite quotient $q_1:G_u \rightarrow Q_1$ where $q_1(g)\notin q_1(G_e)$. Let ${X_u}^1$ be the finite-sheeted covering of $X_u$ that corresponds to $\ker q_1$. Since $q_1(G_e)\subsetneq Q_1$, the preimage of $X_e$ in the cover ${X_u}^1$ has at least two connected components. Since $G_u$ is not virtually cyclic, one may repeat this procedure for one of the elevations of $G_e$ to ${G_u}^1$. After $\ell$ times one obtains a finite cover ${X_u}^\ell$ of $X_u$ in which the preimage of $X_e$ has at least $\ell$ connected components. By Lemma \ref{vertex_sep} there is a finite cover $\hat{X}$ of $X$ and $\hat{u} \in \mathrm{V}(\hat{\Xi})$ with $\hat{X}_{\hat{u}}={X_u}^\ell$; note that $\deg(\hat{u})\ge \ell$ in $\hat{\Xi}$. Therefore, if $\hat{\Xi}$ is a tree, $\hat{\Xi}-\{\hat{u}\}$ has at least $\ell$ connected components. Each of these connected components must contain a vertex that lies above $v$ that is connected to $\hat{X}_{\hat{u}}={X_u}^\ell$ by an edge space that lies above $X_e$. In particular, $\hat{\Xi}$ contains at least $\ell$ vertices whose corresponding vertex group has a positive $n$-th Betti number. As before, employing Mayer-Vietoris we obtain that $b_n^k(\hat{G})\ge \ell$. Finally, if $\hat{\Xi}$ is not a tree, the first part of the proof gives the desired result. 
	\end{proof}

We will also make use of the following corollary: 

\begin{cor} \label{star-tree}
    Suppose that a finitely generated and subgroup separable group $G$ splits as a finite and connected graph of spaces $X$ with infinite cyclic or trivial edge groups. Let  $G_v$ be a vertex group of $\mathcal{G}(X)$ that has a neighbouring vertex group $G_u$ that is not virtually cyclic. Then, for every $\ell \in \mathbb{N}$, there is a finite-index subgroup $H\le G$ with the following properties:
    \begin{enumerate}
        \item  $H$ admits a splitting as a graph of spaces $Y$ whose underlying graph $\Upsilon$ contains a subtree $T\subset \Upsilon$ with $\ell$ or $\ell+1$ vertices,
        \item for every $u\in \mathrm{V}(T)$ (except for one vertex if $\vert \mathrm{V}(T) \vert = \ell + 1$), the vertex group $H_u$ is isomorphic to a finite-index subgroup of $G_v$.
    \end{enumerate}
\end{cor}
 
\subsection{Limit groups and residually free groups}

\label{limintro}
    A group $G$ is called \emph{residually free} if for every $1\ne g \in G$ there is a homomorphism $f:G\rightarrow F$ where $F$ is a free group and $f(g)\ne 1$. $G$ is called \emph{fully residually free} if for every finite subset $E\subset L$ there is a homomorphism from $L$ to a free group whose restriction to $E$ is injective. \par \smallskip
	A \emph{limit group} $L$ is a finitely generated and fully residually free group. Groups with this property have been extensively studied since the 1960s, and were popularized by Sela as they played an integral role in his solution to Tarski's problem about the first-order theory of non-abelian free groups. Limit groups have a rich structural theory and they are exactly the finitely generated subgroups of fundamental groups of \emph{$\omega$-residually free towers}. Loosely speaking, an $\omega$-residually free tower is a space obtained by repeatedly attaching simple building blocks to the wedge sum of circles, surfaces and tori. These simple building blocks are of one of two kinds: tori and surfaces with boundary. We will not use this structure of limit groups explicitly, and we refer the reader to Sela's work \cite{sela2} for the precise definition of an $\omega$-residually free tower. This structure theorem implies that limit groups admit a \emph{hierarchical structure} which we will use in the proof of Theorem \ref{mainlim}.
	
	\begin{defn}
		\label{hierarchy}
		A \emph{cyclic hierarchy} of a group $G$ is a set $\mathcal{H}(G)$ of subgroups of $G$ obtained by iterating the following procedure, starting with $G$: for $H\in \mathcal{H}(G)$, if $H$ admits a splitting $\mathcal{G}(H)$ with (possibly trivial) cyclic edge groups, add each of the vertex groups of this splitting to $\mathcal{H}(G)$.
	\end{defn}
	
	A priori the process described in Definition \ref{hierarchy} above does not necessarily terminate. If this process comes to a halt after finitely many steps (in which case $\mathcal{H}$ is finite), $G$ is said to have a \emph{finite hierarchy}. The groups appearing at the bottom of the hierarchy, that is groups which do not split over a cyclic subgroup, are called \emph{rigid}. Sela showed in \cite{sela1} that if $G$ is a limit group then this process always ends after finitely many steps:
	
	\begin{thm} \label{hierarchy_thm}
		Limit groups have a finite hierarchy. Moreover, rigid groups appearing in the hierarchy of a limit group are free abelian.
	\end{thm}
	
	Sela also showed that limit groups without $\mathbb{Z}^2$ subgroups are hyperbolic \cite[Corollary 4.4]{sela1}, which implies the following:
	
	\begin{cor}
		\label{hierarchy_cor}
		If $L$ is a one-ended limit group and $H\in \mathcal{H}(L)$ is a one-ended group with no one-ended groups below it in the hierarchy, then one of the following holds:
		\begin{enumerate}
			\item $H$ is rigid, and therefore free abelian, or
			\item $H$ is an HGFC-group.
		\end{enumerate}
	\end{cor}

    We also record the following simple consequence of commutative transitivity, which will aid the proof of Theorem \ref{mainlim}:

    \begin{lemma}
        \label{abelian_vtx}
        Suppose that a limit group $L$ contains a subgroup that splits as a cyclic amalgamation $G_1 \ast _{c_1=c_2} G_2$ or as an HNN extension $G\ast_\phi$ (where $\phi$ is an isomorphism between two cyclic subgroups $\langle c_1 \rangle$ and $\langle c_2 \rangle$ of $G$). Then at least one of $\langle c_1 \rangle$ and $\langle c_2 \rangle$ is maximal abelian in its target vertex group.
    \end{lemma}

    %\begin{proof}
     %   Assume that $L$ contains a subgroup of the form $G_1 \ast_{c_1=c_2} G_2$; the HNN extension case is similar (and also follows from the proof for cyclic amalgamation: the HNN extension $G\ast _\phi$ contains a subgroup of index $2$ which has a subgroup of the form $G\ast_{c_1=c_2} G$). \par \smallskip
     %   Suppose for a contradiction that both $\langle c_1 \rangle$ and $\langle c_2 \rangle$ are not maximal abelian in $G_1$ and $G_2$ respectively (or in $G$ in the case of an HNN extension). Let $c'_i$ be an element that lies outside of $\langle c_i \rangle$ and commutes with $c_i$. It is a well known fact that limit groups are commutative transitive, and since both $c'_1$ and $c'_2$ commute with $c_1=c_2$ in $L$ it follows that $c'_1$ and $c'_2$ commute. However, since both of them do not lie in $\langle c_1 \rangle = \langle c_2 \rangle$, the commutator $c'_1 c'_2 {c'_1}^{-1} {c'_2}^{-1}$ is a reduced word in the amalgamated product and is therefore non-trivial, a contradiction.
    %\end{proof}

    As apparent from the above discussion, the structure of limit groups is generally well-understood. The structure of residually free groups is by and large less tame than that of limit groups; however, finiteness properties have strong implications on the structure of residually free groups. We record three useful structural results about residually free groups, all of which are closely intertwined with the finiteness properties of the group in question.
    
    \begin{prop}[Corollary 19 of \cite{Bau99} or Claim 7.5 of \cite{sela1}]\label{subprod} Let $G$ be a finitely generated residually free group. Then $G$ is a subdirect product of finitely many limit groups.
\end{prop}

    \begin{thm}[{\cite[Corollary 1.1]{Bri09}}]
        \label{vir_direct_prod}
        Every residually free group of type $\mathrm{FP}_\infty(\mathbb{Q})$ is virtually a direct product of a finite number of limit groups.
    \end{thm}

    \begin{thm}[Propositions 3.1 and 6.4, \cite{Bri09}] \label{brihow} Let $G$ be a finitely presented subdirect product of limit groups $L_1\times \cdots \times L_n$, with $G \cap L_i \ne 1$ for all $1\le i \le n$. Then there exists a finite-index subgroup $E\leq L_1\times \cdots \times L_n$ and a positive integer $N$ such that $\ga_N E\leq E\cap G$.
\end{thm}

\subsection{Profinite groups and separability properties}
Separability properties allow us to separate subgroups in finite quotients and hence are useful for reading off particular features of a group from its profinite completion. In this section we will review two such properties ({\it LERFness}, in \Cref{LERFness}; and having {\it virtual retractions}, in \Cref{vretractions}) and summarise to what extent do limit groups, residually free groups and HGFC-groups enjoy such properties.

	\begin{defn}\label{LERFness}
		A subgroup $H$ of a group $G$ is called \emph{separable} (in $G$) if it is the intersection of finite-index subgroups of $G$, that is $H=\bigcap \{G_0 \vert H\le G_0\text{ and $G_0\le G$ of finite index}\}$. $G$ is called \emph{subgroup separable}, or \emph{LERF} (locally extended residually finite), if every finitely generated subgroup of $G$ is separable in $G$.
	\end{defn}
	Subgroup separability is a strong tool for generating finite-index subgroups of a given group, as evident from the following elementary lemma which will prove to be useful later on:
	\begin{lemma}
		\label{vertex_sep}
		Let $G$ be a subgroup separable group and let $H$ be a finitely generated subgroup of $G$. Then for every finite-index subgroup $H_0\le H$, there is a finite-index subgroup $G_0\le G$ such that $G_0\cap H=H_0$.
	\end{lemma}
	\begin{proof}
		$H_0$ is a finitely generated subgroup of $G$ and therefore separable in $G$. Let $g_1,\ldots,g_n$ be coset representatives of all non-trivial cosets of $H_0$ in $H$, and for every $i\le n$ let $G_i$ be a finite-index subgroup of $G$ which contains the subgroup $H_0$ but  not the element $g_i$. Letting $G_0=\bigcap_{i=1}^n G_i$ we have that $G_0\cap H=H_0$.
	\end{proof}
	
	Another powerful tool for generating finite-index subgroups with certain homological properties are \emph{virtual retracts}.
	\begin{defn} \label{vretractions}
		Let $G$ be a group and let $H\le G$. We say that \emph{$G$ virtually retracts} onto $H$ if there is a finite-index subgroup $G_0$ of $G$ containing $H$ and a retraction $r:G_0\rightarrow H$, that is a homomorphism $r$ such that $r(h)=h$ for every $h\in H$. In this case we say that $H$ is a \emph{virtual retract} of $G$. If $G$ retracts onto all of its finitely generated subgroups we say that $G$ admits \emph{local retractions}.
	\end{defn}
	
	\begin{rmk}
		\label{vr_sub}
		If $G$ is a residually finite group that admits local retractions then $G$ is subgroup separable. In fact, if $G$ is residually finite and $H$ is a virtual retract of $G$ then $H$ is separable in $G$ (see \cite[Lemma 2.2]{ashotvr}). In addition, if $H\le G$ and $G$ admits local retracts then so does $H$.
	\end{rmk}
	
	The following lemma reduces the problem of showing that $\mathrm{vb}_n^k(G)=\infty$ to finding finitely generated subgroups of $G$ with an arbitrarily large $n$-th Betti number, as long as $G$ admits local retractions.
	
	\begin{lemma}
		\label{vr_red}
		Let $G$ be a finitely generated group and let $H$ be a virtual retract of $G$. Then for any $n$ and any field $k$ there is a finite-index subgroup $G_0\le G$ such that $b_n^k(G_0) \ge b_n^k(H)$.
	\end{lemma}
	
	\begin{proof}
		Let $G_0$ be a finite-index subgroup of $G$ which retracts onto $H$; denote the retraction by $r:G_0\rightarrow H$ and denote by $i:H\rightarrow G_0$ the inclusion map. Note that $r\circ i:H\rightarrow H$ is the identity map. It follows that the same holds for the induced maps on the $n$-th homology, that is $r_* \circ i_* = \mathrm{Id}_{H_n(H;k)}$. In particular, $i_*:H_n(H;k)\rightarrow H_n(G_0;k)$ is injective and $b_n^k(G_0)\ge b_n^k(H)$.
	\end{proof}
	
	We continue by mentioning a few separability properties of HGFC-groups, limit groups and residually free groups. Wise showed in \cite{wise-sep} that HGFC-groups are subgroup separable. Hsu and Wise went on and showed in \cite{hsu-wise} that such groups are in fact virtually compact special (the definition of a special group is highly technical and is beyond the scope of this paper; we therefore refer the reader to Haglund and Wise's work \cite{hag-wise}). The benefits that follow from being compact special will play a key role in the proof of Theorem \ref{mainlim}:

	\begin{thm}[follows from {\cite[Corollary 6.7 and Theorem 7.3]{hag-wise}}]
		Hyperbolic compact special groups virtually retract onto their quasiconvex subgroups.
	\end{thm}

    If we combine the latter result with the fact that every HGFC-group is locally quasiconvex \cite[Theorem D]{quasicon}, we get the following.

	\begin{cor}
		\label{graph_vr}
		Let $G$ be an HGFC-group. Then $G$ has a finite-index subgroup which admits local retractions.
	\end{cor}

	Wilton showed that the same assertion holds for limit groups:
	
	\begin{thm}[{\cite[Theorems A and B]{wilton-hall}}] \label{lim_sep}
		Limit groups are subgroup separable and admit local retractions.
	\end{thm}
	
    Similar strong separability properties also apply to residually free groups, under additional finiteness assumptions. 

\begin{thm}[{\cite[Theorems A and B]{vir_retract_res}}] \label{sep} \label{separable}
    Let $G$ be a finitely generated residually free group, and let $H\le G$. If $H$ is finitely presented, then $H$ is separable in $G$; if furthermore $H$ is of type $\mathrm{FP}_{\infty}(\mathbb{Q})$ then $G$ virtually retracts onto $H$.
\end{thm}

\begin{rmk} 
For the purpose of applying \Cref{separable} in the proof of \Cref{mainlim}, we should note that a direct product $F\times \Z$ of a free group and an abelian group is $\FP_{\infty}(\Q)$ (because it is the extension of the two $\FP_{\infty}(\Q)$ groups $F$ and $\Z$).\end{rmk}

We seal the discussion by mentioning that the proof of Theorem \ref{sep} relies on Theorem \ref{brihow}, which allows one to treat a finitely presented residually free group as if it were a subdirect product of limit groups that contains a term of the lower central series. This idea will also play a key role in the proof of \Cref{productsintro} when detecting direct factors of residually free groups from the pro-$p$ completion.

\subsection{$L^2$ Betti numbers} \label{l2}
We write $\b(G)$ to denote the first $L^2$-Betti number of a group. We will not require much familiarity with its exact definition. Here we recall some of its basic properties that will become useful later. The reader is referred to Lück's book \cite{Luc02} for a more thorough treatment of this topic. We begin by stating a particular instance of Lück's pioneering result on the approximation of $L^2$-Betti numbers by usual Betti numbers \cite{Luc94}.

\begin{thm} \label{Luck} Let  $G$ be a group of type $\FP_2(\Q)$ and let $G=N_1>N_2>\dots >N_m>\dots $ be a sequence of finite-index normal subgroups  with $\bigcap_m N_m=1$. Then 
\[\lim_{m\rar \infty} \frac{b_1(N_m)}{|G:N_m|}= b_1^{(2)}(G).\]
\end{thm}

\begin{prop}\label{limb} Let $G$ be a finitely presented group with $\cd(G)\leq 2$. Then $b_1(G)-b_2(G)\geq 1+ \b(G)$. 
\end{prop}
\begin{proof}  If $G$ is trivial, the claim follows. Otherwise, $b_0^{(2)}(G)=0$ and hence the $L^2$-Euler characteristic of $G$ is equal to $b_2^{(2)}(G)-b_1^{(2)}(G)$. The trivial  $\Z[G]$-module $\Z$ admits a finite free resolution (see \cite[Chapter VIII, Theorem 7.1]{Bro82} and references therein). Thus, as explained in \cite[Theorem 2]{Eck96}, the $L^2$-Euler characteristic of $G$ coincides with the usual Euler characteristic. So $b_2^{(2)}(G)-\b(G)=\chi(G)=b_2(G)-b_1(G)+1.$
\end{proof}
When $G$ is a torsion-free group that satisfies the Strong Atiyah conjecture, Linnell \cite{Lin93} showed that the group ring $\Q G$ is embedded in a division ring $\D_{\Q G}$ (which we will refer to as the {\it Linnell division ring}) and that the usual $L^2$ Betti numbers can be defined as follows. The embedding $\Q G\hookrightarrow \D_{\Q G}$ endows the Linnell division ring $\D_{\Q G}$ with a  $(\D_{\Q G}, \Q G)$-bimodule structure. This induces a natural $\D_{\Q G}$-module structure on the homology groups $H_i(G; \D_{\Q G})$. We can recover $b_i^{(2)}(G)$ from the $\D_{\Q G}$-dimension of $H_i(G; \D_{\Q G})$ (see  \cite[Section 10]{Luc02} for more details). Locally indicable groups satisfy the Strong Atiyah conjecture by the work of Jaikin-Zapirain and López-Álvarez \cite{And192} (which, obviously, includes limit groups). With the above viewpoint on $L^2$-homology,  Jaikin-Zapirain showed the following in \cite{And19}. 

\begin{prop}  \label{BettiAnd} Let $G$ be a finitely generated residually-(locally indicable and amenable) group. Then $\b(G)\leq b_1(G)-1$.
\end{prop}
From here, we introduce a new (virtual homology)-type invariant.
\begin{defn}\label{vd} Let $G$ be a finitely generated group. We define its {\it first $L^2$-difference} by $\d(G)=b_1(G)-\b(G)$ and its {\it virtual $L^2$-difference} by $$\vd(G)=\sup\{ \d(H) \, \vert \, H \text{ is a finite-index subgroup of } G\}.$$
\end{defn}

By \Cref{BettiAnd}, for residually-(amenable and locally indicable) groups, $\d(G)\geq 1$. In \Cref{virtual_Bettis}, we show that this property is very particular (at least, among limit groups, it is only satisfied by free and surface groups). For this aim, we finish this section by giving two  estimations of $\d$ in the context of cyclic group hierarchies. 

\begin{lemma} \label{L^2amalgam}
    Consider two finitely generated torsion-free groups $G_1$ and $G_2$ with elements $a_1\in G_1$ and $a_2\in G_2$ that generate isomorphic  groups $\lan a_1\ran\cong \lan a_2\ran$.  Denote by $G$  the corresponding amalgamated product $G_1*_{a_1=a_2}G_2$. Then \[\d(G)\geq \d(G_1)+\d(G_2)-1.\] 
    In particular, if $G$ is residually-(amenable and locally indicable), then $\d(G)\geq \d(G_1).$
\end{lemma} 
Before the proof we note that, developing on the above, given a subgroup $H\leq G$ of a torsion-free group satisfying the Atiyah conjecture, it can be seen that $\D_{\Q H}$ is naturally embedded in $\D_{\Q G}$ and that, in particular, $b_i^{(2)}(H)=\dim_{\D_{\Q H}} H_i(H; \D_{\Q H})=\dim_{\D_{\Q G}} H_i(G; \D_{\Q G}).$ 

\begin{proof}[Proof of \Cref{L^2amalgam}] The second part  is a direct consequence of \Cref{BettiAnd}. For the first inequality, we use the Mayer–-Vietoris sequence for amalgamated free products in group cohomology \cite[Theorem 2.3]{Swa69}. Denote by $A\leq G_1$ the cyclic group generated by $a_1$. We have the following exact sequence of $\D_{\Q G}$-modules. 
\begin{equation*}
    \begin{tikzpicture}[scale=2]
\matrix(m)[matrix of math nodes,column sep=15pt,row sep=15pt]{
   &  &  & H_2(G; \D_{\Q G}) \\
    & H_1(A; \D_{\Q G}) & H_1(G_1; \D_{\Q G})\oplus H_1(G_2; \D_{\Q G}) & H_1(G; \D_{\Q G}) &  \\
    & H_0(A; \D_{\Q G}) & H_0(G_1; \D_{\Q G})\oplus H_0(G_2; \D_{\Q G}) &  \\
};
\draw[->,font=\scriptsize,every node/.style={above},rounded corners]
  (m-1-4.east) --+(5pt,0)|-+(0,-7.5pt)-|([xshift=-5pt]m-2-2.west)--(m-2-2.west)
  (m-2-2) edge (m-2-3)
  (m-2-3) edge (m-2-4)
  (m-2-4.east) --+(5pt,0)|-+(0,-7.5pt)-|([xshift=-5pt]m-3-2.west)--(m-3-2.west)
  (m-3-2) edge (m-3-3)
;
\end{tikzpicture}
\end{equation*}
Since each $G_1$ and $G_2$ are infinite, $H_0(G_1; \D_{\Q G})=H_0(G_2; \D_{\Q G})=0$. In addition, since $A$ is cyclic, $H_1(A; \D_{\Q G})=0$ as well. Now we divide in two cases.
If $A$ is trivial, then  $H_0(A; \D_{\Q G})\cong \D_{\Q G}$ and hence $\b(G)=\b(G_1)+\b(G_2)+1$ and $b_1(G)= b_1(G_1)+b_1(G_2).$ The conclusion follows in this case. Lastly, if $A\cong \Z$, then $H_0(A; \D_{\Q G})=H_1(A; \D_{\Q G})=0$ and so $\b(G)=\b(G_1)+\b(G_2).$ It is also easy to see that $b_1(G)\geq b_1(G_1)+b_1(G_2)-1.$ This completes the proof. 
\end{proof}

 The statement of \Cref{L^2amalgam} and its proof can be easily adapted to give the following HNN version (this time, using the Mayer-Vietories sequence for HNN extensions in group cohomology \cite{Bie75}).
\begin{lemma} \label{L^2hnn}
    Let $H$ be a finitely generated torsion-free group with a monomorphism $\theta \po A\hookrightarrow H$, where $A$ is a cyclic group (possibly trivial). Consider the corresponding  HNN extension $ G= H *_{A, \theta}$. Then $\d(G)\geq \d(H)$. 
\end{lemma} 

Combining Lemmas \ref{L^2amalgam} and \ref{L^2hnn} with an induction, both lemmas can be included in the following more general statement about HGFC-groups. 

\begin{prop} \label{L^2general} Let $(\G, Y)$ be a  finite connected graph of finitely generated groups with cyclic edges and suppose that its fundamental group $\pi_1(\G, Y)$ is residually-(amenable and locally indicable). Let $Z\sub Y$ be a connected subgraph. Then \[\d(\pi_1(\G, Y))\geq \d(\pi_1(\G, Z))\geq \sum_{v\in Z} \d(\G(v))-(|V(Z)|-1).\]
\end{prop}
We will use these estimations on $L^2$-Betti numbers to prove in \Cref{boundedvd} that the only limit groups $G$ with finite $\vd(G)$ are free abelian, free and surface groups.

\section{Classifying limit groups by their virtual homology} \label{SurfaceSection}

In this section we prove Theorem \ref{mainlim}. We begin by amassing the ingredients required for the proof. The main idea behind the proof that in some instances, the existence of a subgroup $H$ of $G$ with a positive $n$-th Betti number can be promoted to the existence of a subgroup $H' =H \ast \mathbb{Z}$ of $ G$. Finally, finite index subgroups of $H'$ contain multiple isomorphic copies of $H$, independent from each other in homology. This results in subgroups of $G$ with a large $n$-th Betti number. Using local retractions, we obtain a finite-index subgroup of $G$ with large $n$-th homology, which results in $\mathrm{vb}_n^k(G)=\infty$. Consequently, the following theorem of Wilton's will play a crucial role in the proof of Theorem \ref{mainlim}:

\begin{thm}[{\cite[Theorem 6.1]{surfacesubs}}] 
		Let $G$ be an HGFC-group. If $G$ is one-ended then $G$ contains a surface subgroup.
	\end{thm}

The following theorem of Arzhantseva's will help us in replicating the surface subgroups provided by Wilton's Theorem above:

\begin{thm}[{\cite[Theorem 1]{Arzhantseva2001}}] \label{quasiconfree}
    Let $G$ be a non-elementary and torsion-free hyperbolic group, and let $H$ be a quasiconvex subgroup of $G$ of infinite index. Then there exists $g\in G$ such that $\langle H, g\rangle \cong H \ast \mathbb{Z}$. Furthermore, $\langle H, g \rangle $ is quasiconvex in $G$.
\end{thm}
    Another useful fact, is that every HGFC-group splits as the free product of one-ended (hyperbolic) graphs of free groups with cyclic edge groups and a free group in the following sense:

    \begin{lemma}[follows from Shenitzer's Lemma {\cite[Theorem 18]{wilton-one-ended}}]
        \label{graph_grushko} Let $G$ be a finitely generated group which splits as a (not necessarily hyperbolic) graph of free groups with cyclic edge groups. Then $G\cong G_1 \ast \cdots \ast G_n \ast F_r$ where each $G_i$ is a one-ended fundamental group of a graph of free groups with cyclic edge groups, and $F_r$ is a free group of rank $r\ge 0$. Furthermore, if $G$ is an HGFC-group, then every $G_i$ is an HGFC-group.
    \end{lemma}

  %  \begin{proof}
%        By Grushko's theorem we may write $G=G_1 \ast \cdots \ast G_n \ast F_r$ where each $G_i$ is one-ended and $F_r$ (which might not appear in this decomposition, that is we may have $r=0$) is free. A standard argument using Bass-Serre theory shows that in fact each $G_i$ is the fundamental group of a finite graph of finitely generated free groups with infinite cyclic edge groups: let $T$ be the Bass-Serre tree that corresponds to the cyclic splitting of $G$ and let $T_i$ be a minimal $G_i$-invariant subtree of $T$. Taking the core of the quotient of $T_i$ by $G_i$ we obtain a finite graph of groups decomposition $\mathcal{G}(G_i)$ of $G_i$. Since $G_i$ is freely indecomposable, the edge groups of $\mathcal{G}(G_i)$ are all infinite cyclic; in particular, these edge groups are finitely generated. By Grushko's theorem $G_i$ is finitely generated which implies that the vertex groups of $\mathcal{G}(G_i)$ are finitely generated. In addition, since the vertex groups of the cyclic splitting of $G$ are free, the vertex groups of $\mathcal{G}(G_i)$ are all free. \par \smallskip
  %      Lastly, if $G$ is hyperbolic, then by \cite[Theorem D]{quasicon} $G$ is locally quasiconvex, which implies that every $G_i$ is hyperbolic.
  %  \end{proof}

    The following lemma will help us in turning cyclic amalgamations of a group $H$ with a free group into subgroups of the form $H \ast \mathbb{Z}$. This will be used when the group $G$ in question is a non-hyperbolic limit group. The proof boils down to looking at normal forms, and is left as an exercise.

    \begin{lemma} \label{free_prod_sub}
    Let $G$ be a group and let $F$ be a non-abelian free group. Then any cyclic amalgamation $G\ast_{c=c'}F$ contains a copy of $G\ast \mathbb{Z}$.
    \end{lemma}

    We are finally ready to prove Theorem \ref{mainlim}:

    \begin{thm}[Theorem \ref{mainlim}] \label{mainlim2}
    Let $G$ be a finitely generated residually free  group or an HGFC-group. Then
    \begin{enumerate}
			\item $\mathrm{vb}_2^k(G)=0$ if and only if $G$ is free,
			\item $\mathrm{vb}_2^k(G)=1$ if and only if $G \cong \pi_1(\Sigma)$ where $\Sigma$ is a closed, connected surface,
            \item $\mathrm{vb}_2^k(G)=\binom{d}{2}$ if and only if $L\cong \mathbb{Z}^d$ (for $d>2$)
			\item $\mathrm{vb}_2^k(G)=\infty$ otherwise.
		\end{enumerate}

    Furthermore, if $G$ is a limit group, then for any $n\ge 3$ we have that $\mathrm{vb}_n^k(L)=\infty$ unless one of the following two holds:
    \begin{enumerate}
        \item $\mathrm{cd}(L)<n$, in which case $\mathrm{vb}_n^k(L)=0$, or
        \item $L$ is free abelian of rank at least $n$, in which case $\mathrm{vb}_n^k(L)=\binom{\mathrm{rank}(L)}{2}$.
    \end{enumerate}
    \end{thm}
    \begin{proof}[Proof of Theorem \ref{mainlim}]
    Suppose first that $G$ is a residually free group that is not a limit group. Baumslag  \cite{Bau67a} showed that, in this case, $G$ contains a subgroup isomorphic to $F\times \mathbb{Z}$, where $F$ is a free group of rank $2$. It follows that $G$ contains a subgroup isomorphic to $F_n \times \mathbb{Z}$ (where $F_n$ is a free group of rank $n$), and by K\"unneth formula $H_1(F_n;k)\otimes H_1(\mathbb{Z};k)=k^n\otimes k\cong k^n$ embeds in $H_2(F_n\times \mathbb{Z};k)$ so $b_2^k(F_n\times \mathbb{Z})\ge n$.\par \smallskip

   Both $F_n$ and $\mathbb{Z}$ are of type $\mathrm{FP}_\infty(\mathbb{Q})$, so the group $F_n\times \mathbb{Z}$ is also of type $\mathrm{FP}_\infty(\mathbb{Q})$  \cite[Proposition 2.7]{Bie87}. By \cite[Theorem A]{vir_retract_res}, $F_n\times \mathbb{Z}$ is a virtual retract of $G$, and hence Lemma \ref{vr_red} leads to $\mathrm{vb}_2^k(G)\ge n$. \par \smallskip

   We next suppose that $G$ is a limit group or an HGFC-group, and that $G$ is not a free, a surface or a free abelian group. Recall that if $G$ is free then $\mathrm{vb}_2^k(G)=0$, if $G$ is a surface group then $\mathrm{vb}_2^k(G)=1$ and if $G\cong \mathbb{Z}^d$ then $\mathrm{vb}_n^k(G)=\binom{d}{n}$ for every $n\in \mathbb{N}$. By passing to a finite-index subgroup, we may assume that $G$ admits local retractions (since $G$ is torsion-free and not isomorphic to a free, a surface or a free abelian group, it cannot have a finite-index subgroup that is free, free abelian or surface). We will find a subgroup $H$ of $G$ with an arbitrarily large second (or $n$-th, if $\mathrm{cd}(L)\ge n$) Betti number.
   
   Suppose first that $G$ is hyperbolic, that is $G$ is either an HGFC-group or a hyperbolic limit group. If $G$ is one-ended, and $G$ is not a surface group, then by \cite[Theorem 6.1]{surfacesubs} $G$ contains a surface subgroup $H \cong \pi_1(\Sigma)$. Note that $H$ is quasiconvex in $G$: if $G$ is an HGFC-group the fact that $H$ is quasiconvex follows from \cite[Theorem D]{quasicon}, and if $G$ is a hyperbolic limit group then by \cite[Corollary 3.12]{wilton-hall} $H$ is quasi-isometrically embedded in $G$ which implies that it is quasiconvex. By Theorem \ref{quasiconfree} \cite[Theorem 1]{Arzhantseva2001}, $G$ contains $H\ast \mathbb{Z}$ as a subgroup \cite[Theorem 1]{Arzhantseva2001}. \par \smallskip

    If $G$ is not one-ended, then by Lemma \ref{graph_grushko} we may write $G=G_1 \ast \cdots \ast G_n \ast F_r$ where each $G_i$ is a one-ended HGFC-group, $F_r$ (which might not appear in this decomposition) is free and there are at least $2$ factors in this free product. If some $G_i$ is not a surface group, then the previous paragraph implies that $G_i$, and hence $G$, contains a subgroup of the form $\pi_1(\Sigma)\ast \mathbb{Z}$. Similarly, if some $G_i$ is the fundamental group of a closed, connected hyperbolic surface, then $G$ clearly contains a subgroup of the desired form. \par \smallskip

    We last consider the case where $G$ is a non-hyperbolic limit group; in this case, there is a free abelian group $H$ of rank at least $2$ appearing at the bottom of the hierarchy of $G$; in fact, employing the hierarchical structure of $G$, we may choose $H$ to be of rank $\mathrm{cd}(G)$. Since $G$ is not a free abelian group, there is a group $H'$ lying above $H$ in the hierarchy. If $H$ is a free factor of $H'$, then $G$ contains a subgroup of the form $H\ast \mathbb{Z}$ as desired (note that a free abelian group of rank $2$ is the fundamental group of the torus $T^2$, and $\chi(T^2)=0$). If not, then $H'$ contains a free factor $K$, which admits a cyclic splitting in which $H$ is one of the factors; denote this splitting by $\mathcal{G}_K$. By Lemma \ref{abelian_vtx}, $H$ cannot be the only vertex group of $\mathcal{G}_K$, and there is a vertex group $J$ of $\mathcal{G}_K$ that is adjacent to $H$. Moreover, Lemma \ref{abelian_vtx} implies that $J$ is non-abelian. Since any two non-commuting elements of $J$ must generate a free group of rank $2$, $K$ contains a subgroup of the form $H\ast_{c=c'}F$ (where $F$ is a free group of rank $2$). Lemma \ref{free_prod_sub} gives rise to an embedding $H \ast \mathbb{Z}\hookrightarrow G$. \par \smallskip

    In each of these cases, given $n\in \mathbb{N}$, the kernel of the map $f:H\ast \mathbb{Z}\rightarrow \mathbb{Z}/m\mathbb{Z}$ which kills $H$ and maps $\mathbb{Z}$ onto $\mathbb{Z}/m\mathbb{Z}$ is isomorphic to $H \ast H \ast \cdots \ast H \ast \mathbb{Z}$ (where $H$ appears $m$ times in the product). A repeated application of Mayer-Vietoris shows that as long $\mathrm{b}_n^k(H)>0$ then the $n$-th Betti number of $\ker f$ (with coefficients in $k$) is at least $m\cdot \mathrm{b}_2^k(H)\ge m$. Lemma \ref{vr_red} now implies that $\mathrm{vb}_2^k(G)=\infty$, and more generally that $\mathrm{vb}_n^k(G)=\infty$ for any $n\le \mathrm{cd}(G)$. 
    \end{proof}

    \subsection{Residually free manifolds}

    A natural class of groups whose $n$-th virtual Betti number is always finite (and equal to $1$) is the class of fundamental groups of aspherical, closed manifolds of dimension $n$. We seal the section by deducing the following, using the classification of limit group by their $n$-th virtual number in Theorem \ref{mainlim} above. We remark that Wilton classified all compact $3$-manifolds with incompressible toral boundary and a residually free fundamental group in \cite{wilton3m}.

    \begin{cor}[Corollary \ref{mflds}] \label{mflds2}
    Let $M$ be an aspherical, closed manifold of dimension $n$. Then the following two hold:
    \begin{enumerate}
        \item if $\pi_1(M)$ is fully residually free and $n\ge 3$ then $M \cong T^n$.
        \item if $\pi_1(M)$ is residually free and $n\ge 5$ then $M$ has a finite cover that is homeomorphic to the direct product of a torus and finitely many closed surfaces.
    \end{enumerate}
    \end{cor}

    \begin{proof}[Proof of Corollary \ref{mflds}]
    The manifold $M$ is aspherical, and so are all its finite covers. By Poincaré duality, it follows that $\vb_n(M)=1$. \par \smallskip
    If $\pi_1(M)$ is fully residually free, for $n\ge 3$ Theorem \ref{mainlim} implies that $\pi_1(M)\cong \mathbb{Z}^n$. Note that the classifying map for $\pi_1(M)$, $f:M \rightarrow K(\mathbb{Z}^n,1)=T^n$ is a homotopy equivalence, and in particular $M$ is orientable. \par \smallskip
    If $n=3$, then since $\pi_1(M)\cong \mathbb{Z}^n$ is one-ended, the Poincar\'e Conjecture implies that $M$ is prime, and therefore irreducible. Waldhausen's homeomorphism theorem \cite{wald} we have that $M \cong T^3$. For $n=4$, since $\mathbb{Z}^4$ is good in the sense of Freedman, \cite[Section 11.5]{4mflds} implies that $M \cong \mathbb{Z}^4$. Lastly, for $n\ge 5$, by \cite[Theorem A]{borel} the Borel Conjecture holds for $\mathbb{Z}^n$, so $M$ is homeomorphic to $T^n$. \par \smallskip
    
    Suppose now that $\pi_1(M)$ is residually free. By \cite{CW} $M$ has the homotopy type of a finite CW-complex and therefore $\pi_1(M)$ is of type $\mathrm{FP}(\mathbb{Q})$. By Theorem \ref{vir_direct_prod}, $M$ has a finite cover $\hat{M}$ whose fundamental group is a direct product of finitely many limit groups. Since $\vb_n(\pi_1(\hat{M}))=1$ and $\vb_\ell(\pi_1(\hat{M}))=0$ for $\ell>m$, Theorem \ref{mainlim} and a direct computation using K\"unneth formula imply that $\pi_1(\hat{M})$ is the direct product of finitely many surface groups and a free abelian group. Invoking \cite[Theorem A]{borel} yields that $\hat{M}$ is homeomorphic to the direct product of a torus and finitely many closed surfaces, as desired.
\end{proof}

\begin{rmk} \label{mflds_rmk}
    The corollary above can also be proved using other cohomological techniques, and in particular by relying on the fact that a Poincaré duality group of dimension $n$ can not have a Poincaré duality group of dimension $n$ as an infinite-index subgroup.
\end{rmk}

\begin{rmk}
    Passing to a finite cover in (2) of Corollary \ref{mflds} above is necessary: let $G=S_1 \times S_2 \times S_3$ where each $S_i$ is the fundamental group of a surface of genus $2$, and let $H$ be its index two subgroup which is the kernel of the map $G\rightarrow \mathbb{Z} / 2 \mathbb{Z}$ which maps each of the standard generators of $G$ to the non-trivial element in $\mathbb{Z} / 2 \mathbb{Z}$. $H$ is the fundamental group of a $6$-manifold which is a double cover of the product of three surfaces, and it is residually free as a subgroup of a residually free group. Suppose now that $H$ splits as a direct product $H_1 \times H_2$. Projecting $H_1$ and $H_2$ to each of the three factors $S_i$ of $G$, we get that for every $i$ one of $p_i(H_1)$ and $p_i(H_2)$ is trivial in $S_i$. By the pigeonhole principle, either $H_1$ or $H_2$ is contained in one of the $S_i$, say $H_1$ is contained in $S_1 \times \{1\} \times \{1\}$. This implies that $H_2$ is contained in $\{1\} \times S_2 \times S_3$. Since $[G:H]=2$ it follows that either $H_1 = S_1 \times \{1\} \times \{1\}$ or $H_2=\{1\} \times S_2 \times S_3$, which contradicts the fact that $H$ does not contain any of the factors of $G$·
\end{rmk}

\section{On the profinite rigidity of direct products of free and surface groups} \label{ProductsSection}
In this section we introduce the required tools to prove \Cref{productsintro}, which we restate in \Cref{products2}.  
Given a group $G$, we denote by $\hat G$ and $G_{\hp}$ its profinite and pro-$p$ completions respectively. As usual, we denote by $\b(G)$ its first $L^2$-Betti number. Let $\CC$ be the variety of either finite groups, finite nilpotent groups or finite $p$-groups.
Given a finitely generated group $ G$, we say that it is \emph{residually-$\CC$} if, for all $1\neq g\in G$, there exists $N\n G$ such that $g\notin N$ and $G/N \in \CC$.  Following the terminology of \cite{Gru11}, given two residually-$\CC$ groups $G$ and $\Gamma$, we say that $G$ belongs to the \emph{$\CC$-genus} of $\Ga$ if they have the same collection of isomorphism types of finite quotients belonging to $\CC$. Equivalently, this is true if there is an abstract isomorphism $G_{\hat\CC}\cong  \Ga_{\hat\CC}$ between their pro-$\CC$ completions \cite{Dix82}. Recall that $b_1$ and $b_1^{\mathbb{F}_p}$ are profinite and pro-$p$ invariants within any of these genera, which allows us to show that $L^2$-Betti numbers, or their pro-$p$ analogues (which we review below), are detected by the pro-$p$ completion. 

\begin{notation} Given an abstract group $G$, we denote its profinite completion (resp. pro-$p$ completion) by $\hat G$ (resp. $G_{\hp}$). Bold letters such as $\AA, \BB, \GG, \HH$ will denote profinite groups. In addition, we write $\HH\lc \GG$ or $\HH\lo \GG$ (resp. $\HH \nc \GG$ or $\HH \no \GG$) to indicate that the subgroup $\HH\leq \GG$ (resp. normal subgroup $\HH\n \GG$) is closed or open.  
\end{notation} 

 The only profinite groups that are finitely generated as abstract groups are finite (since, otherwise, they are not even countable). So when we say that a profinite group $\GG$ is finitely generated, we really mean that it is topologically finitely generated, i.e. that there exists a finite subset $S\subset \GG$ that generates a dense abstract subgroup of $\GG$. One of the required ingredients of \Cref{productsintro} is \Cref{Melnip}, which is a direct consequence of \Cref{mainlim}. 
 
\begin{thm} \label{Melnip} Let $G$ be a finitely generated, residually free group. Suppose that for all finite-index subgroups $H\le G$, there exists a prime $p$ such that $H_{\hp}$ has a presentation with at most one relator. Then $G$ is either a free or a surface group.
\end{thm}
%\begin{proof} Suppose that $G$ is not a free or a surface group; note in addition that under the assumption, $G$ can not be a free abelian group of rank $n\ge 3$. If $G$ is a limit group, then, by Theorem \ref{mainlim}, $G$ contains a subgroup isomorphic to $S \ast \mathbb{Z}$, where $S$ is the fundamental group of a closed, orientable surface. In particular, $G$ contains a subgroup $A$ isomorphic to $S \ast S$. By Theorem \ref{lim_sep}, $G$ virtually retracts onto $A$, so there exists a finite-index subgroup $H\leq G$ and a retraction $r\po H\lrar A$. $r$ induces a retraction at the level of pro-$p$ completions $r\po H_{\hp}\lrar A_{\hp}$ and, similarly, at the level of their continuous cohomology with trivial $\Z/p$ coefficients $r\po H^2(H_{\hp}; \Z/p)\lrar H^2(A_{\hp}; \Z/p)\cong (\Z/p)^2$. In particular, the dimension of $ H^2(H_{\hp}; \Z/p)$ is at least two and hence $H_{\hp}$ cannot admit a one-relator presentation for any prime $p$. \par \smallskip Lastly, if $G$ is not a limit group, then it contains a subgroup $A=F\times \mathbb{Z}$, where $F$ is a free group of rank $2$. By Theorem \ref{sep}, $A$ is a virtual retract of $G$, and, as before, we obtain that $\dim_{\Z/p} (H^2(H_{\hp}; \Z/p)) \ge 2$, which completes the proof.\end{proof}
\subsection{Nilpotent groups} \label{NilpotentSection}
In this section, $G$ is going to denote  a finitely generated nilpotent group and $\GG$ a topologically finitely generated nilpotent pro-$p$ group. The following proposition is classical and recalls the notion of Hirsch length and its  properties. A more detailed account of the theory of polycyclic groups is given in the book of D. Segal \cite{Seg83}.

\begin{prop} \label{hir} Let $G$ be a finitely generated nilpotent group. Then the following statements hold.
\begin{itemize}
    \item[(a)] There exists a finite subnormal series $1=G_0\n G_1\n \cdots \n G_n=G$ such that each quotient $G_{i+1}/G_i$ is isomorphic to a cyclic group. The number of subscripts $i$ such that $G_{i+1}/G_i$ is infinite is independent of the chosen subnormal series and it is called the {\it Hirsch length} of $G$, denoted by $h(G)$.
    \item[(b)] The Hirsch length is additive in the following sense: for all normal subgroups $N\n G$, we have that $h(G)=h(N)+h(G/N).$
    \item[(c)] Let $H\leq G$ be a subgroup. Then $H$ has finite index in $G$ if and only if $h(H)=h(G)$.
\end{itemize}
\end{prop}

These properties are standard for a more general class, say, the class of virtually polycyclic groups, but we will introduce in \Cref{propimage} a pro-$p$ strengthening of part (c) of \Cref{hir} that only holds for nilpotent groups. The analogous definition of Hirsch length and its properties carries over into the category of pro-$p$ groups.  
\begin{prop} \label{hirp} Let $\GG$ be a topologically finitely generated nilpotent pro-$p$ group. Then the following statements hold.
\begin{itemize}
    \item[(a)]  There exists a finite closed subnormal series $1=\GG_0\nc \GG_1\n \cdots \nc \GG_n=\GG$ such that each quotient $\GG_{i+1}/\GG_i$ is isomorphic to a pro-$p$ cyclic group. The number of subscripts $i$ such that $\GG_{i+1}/\GG_i$ is infinite is independent of the chosen subnormal series and it is called the {\it Hirsch length} of $\GG$, denoted by $h(\GG)$.
    \item[(b)] For all closed normal subgroups $\NN\n \GG$, we have that $h(\GG)=h(\NN)+h(\GG/\NN).$
    \item[(c)]  Let $\HH\lc \GG$ be a closed subgroup. Then $\HH$ has finite index in $\GG$ if and only if $h(\HH)=h(\GG)$.
\end{itemize}
\end{prop}
The main reason we introduced the Hirsch length is the following proposition, which shows that it allows us to detect when a subgroup has finite index.
\begin{prop} \label{propimage} Let $G$ be a finitely generated nilpotent group and let $H\leq G$ be a subgroup. Then $H$ has finite index in $G$ if and only if the image of the induced map on pro-$p$ completions $H_{\hp}\lrar G_{\hp}$ has finite index.
\end{prop}
Before giving the proof, we recall that for any subgroups $H_1\leq H_2$ of $G$, the induced map in pro-$p$ completions $(H_1)_{\hp}\lrar (H_2)_{\hp}$ is injective \cite[Theorem D]{Mor22} (this was also proved in \cite[Lemma 5.1]{Kro20}).
\begin{proof} One implication is direct. If the inclusion $\iota\po H \lrar G$ has finite-index image, then the image of the induced map $\iota_{\hp}\po H_{\hp}\lrar G_{\hp}$, which is equal to the closure $\bar{H}$ of $H$ in $G_{\hp}$,  has finite index in $\bar{G}= G_{\hp}$. To prove the converse, suppose that the induced map on pro-$p$ completions $\iota_{\hp}\po  H_{\hp}\lrar G_{\hp}$ has a finite-index image. By standard theory of polycyclic groups, there exists a finite-index torsion-free subgroup $G'$ of $G$ (see \cite[Chapter 1]{Seg83}). Consider $H'=H\cap G'$. It is enough to show that $H'$ has finite index in $G'$. The induced injective map  $H'_{\hp}\lrar G'_{\hp}$ still has a finite-index image. Hence, by part (c) of \Cref{hirp}, we derive that $h(H'_{\hp})=h(G'_{\hp})$. We know from the proof of \cite[Proposition 2.15]{Mor22} that $h(\Ga)=h(\Ga_{\hp})$ for any finitely generated torsion-free nilpotent group $\Ga$. Hence $h(H')=h(H'_{\hp})=h(G'_{\hp})=h(G')$. Finally, we can conclude that $\iota(H)\leq G$ has finite index by part (c) of \Cref{hir}.
\end{proof}
\begin{rmk} Note that \Cref{propimage} is not true for general virtually polycyclic groups. In fact, one can consider the fundamental group of the Klein bottle $G=\lan a, b \, |\, abab^{-1}\ran$, which is virtually abelian and also residually-$p$ for all primes $p$. The subgroup $H=\lan b\ran$ has infinite index and the induced map $H_{\hp}\lrar G_{\hp}\cong \Z_p$ is an isomorphism for all primes $p>2$.
 \end{rmk}

\subsection{Pro-$p$ groups with positive rank gradient} \label{GradientSection}
Free groups do not have non-trivial infinite-index normal subgroups that are finitely generated; this can be verified by means of covering space theory. This property can be generalised with $L^2$-methods to the class of groups $G$ with $\b(G)>0$. A direct consequence of this property is, for example, that the automorphisms of the direct products $F_n\times F_n$ are "rectified" up to permutation, i.e. $\Aut(F_n\times F_n)\cong \left(\Aut (F_n)\times \Aut (F_n)\right)\rtimes (\Z/2\Z)$. Even if the previous can be shown in greater generality without an $L^2$ assumption, we observe that such properties naturally carry over to the category of pro-$p$ groups by means of the analogous (and even better behaved) notion of \emph{rank gradient}. 

Given a pro-$p$ group $\GG$ and descending chain of normal open subgroups $\GG=\GG_0\unrhd\GG_1 \unrhd \GG_2 \unrhd\cdots$ with trivial intersection, we define (as in \cite{Oca19}) the  {\it rank gradient of $\GG$ relative to $\{\GG_i\}$} by 
 \[\RG (\GG; \{\GG_i\})=\lim_{k\rar \infty}\frac{d(\GG_k)-1}{|\GG: \GG_k|}, \]
 where the previous limit exists because it is monotone decreasing and bounded from below. For the same reason, this relative rank gradient does not depend on the chain and it is equal to the {\it absolute rank gradient} of $\GG$, defined by
 \[\RG(\GG)=\inf_{\UU\no \GG} \frac{d(\UU)-1}{|\GG: \UU|}.\] 
For analogous reasons, we can extend the notion of $p$-rank for  a finitely generated residually-$p$ group $G$ as follows. Let  $G=G_0\unrhd G_1 \unrhd G_2 \unrhd\cdots$ be a chain of normal $p$-power index subgroups such that for all $n$ there exists $N$ such that $\ga_n\sub G_N$. Then  we  define the $p$-gradient of $G$, relative to the chain $\{G_i\}$, by 
\[\RG (G; \{G_i\})=\lim_{k\rar \infty}\frac{\dim_{\F_p} H^1(G_k; \F_p)-1}{|G: G_k|}, \]
  Again, the monotonicity of this chain ensures that the previous limit exists and that it is independent of the chain $\{G_i\}$.
The following is a particular instance of \cite[Theorem 1.1]{Oca19}, which tells us that these pro-$p$ groups do not have intermediate finitely generated normal subgroups. 
\begin{thm} \label{posgrad} Let $\GG$ be a finitely generated pro-$p$ group with positive rank gradient and let $\NN\nc \GG$ be a finitely generated closed normal subgroup. Then $\NN$ is either finite or open.
\end{thm}
We  note that, by Lück's approximation \cite{Luc94},  when $G$ is a finitely presented residually-$p$ group, we have $\b(G)\leq \RG(G_{\hp})$.
So \Cref{posgrad} applies to pro-$p$ completions of finitely presented residually-$p$ groups $G$ with $\b(G)>0$ (which includes non-abelian limit groups).

As we anticipated, \Cref{posgrad} has some "rigidity" consequences on the structure of automorphism groups of direct products of pro-$p$ groups with positive rank gradient exhibited by the following lemma.

\begin{prop} \label{s1} Let $\GG_i$ and $\HH_j$, with $1\leq i\leq n$ and $1\leq j\leq m$, be  finitely generated infinite pro-$p$ groups that are torsion-free and have positive rank gradient. Suppose that there is an isomorphism of pro-$p$ groups
\[\phi\po \GG_1\times \cdots \times \GG_n\lrar \HH_1\times \cdots \times \HH_m.\]
Then $n=m$ and $\phi$ is a direct product of isomorphisms, i.e. there exist a permutation $\sigma\in \mathrm{Sym} (n)$ and isomorphisms $\phi_i \po(G_i)_{\hp}\lrar (H_{\sigma (i)})_{\hp}$ for all $1\leq i\leq n$ such that $\phi=\left(\phi_{\sigma^{-1} (1)}\circ \pi_{\sigma^{-1} (1)}\right)\times \cdots \times \left(   \phi_{\sigma^{-1} (n)}\circ \pi_{\sigma^{-1} (n)}\right)$.
\end{prop} 
\begin{proof} By assumption, we know that each $\GG_i$ and $\HH_j$ are infinite. For both of the groups $\GG_1\times \cdots \times \GG_n$ and  $\HH_1\times \cdots \times \HH_m$, we denote by $\pi_i$ the projection onto the $i$-th coordinate. 

For each $1\leq i \leq n$, we know that there must be $1\leq j\leq m$ such that $\pi_j(\phi(\GG_i))$ is an infinite subgroup of $\HH_j$. A priori, there may be multiple choices of $j$ for a given $i$, but let us take any choice $j(i)$ for every $1\leq i\leq n$. We know that $\pi_{j(i)}(\phi(\GG_i))$ is normal in $H_{j(i)}$ so, by \Cref{posgrad}, $\pi_{j(i)}(\phi(\GG_i))$ is open in $\HH_{j(i)}$. In particular, this shows that we cannot have for different $i_1\neq i_2$ that $j(i_1)=j(i_2)$. Otherwise, if $j=j(i_1)=j(i_2)$, then $\HH_{j}$ would contain the two open commuting subgroups $\pi_{j(i_1)}(\phi(\GG_{i_1}))$ and $\pi_{j(i_2)}(\phi(\GG_{i_2}))$, which would intersect in an open subgroup, implying that $\HH_{j}$ is infinite and virtually abelian (and hence it would not have positive rank gradient). This contradiction shows that $j$ had to be, in particular, an injection $\{1, \dots, n\}\lrar \{1, \dots, m\}$. Hence $n\leq m$. Reasoning analogously for the inverse map $\phi^{-1}$, we would also derive that $m\leq n$. Hence $m=n$ and so $j$ is a bijection. In fact, for the exact same reasons as above, for any $1\leq i\leq n$ and any $j\neq j(i)$, we must have that $\pi_{j}(\phi(\GG_i))=1.$ Lastly, it is also easy to observe that, since $\phi$ is an isomorphism, $\pi_{j(i)}(\phi(\GG_i))=\HH_{j(i)}$. This proves the lemma by just taking $\sigma=j$ and $\phi_i$ to be the restriction of $\pi_{j(i)}\circ \phi$ to $\GG_i$.
\end{proof}

The analogous statement of \Cref{s1} about free groups was first proved by Schreier \cite{Sch27}. In the abstract setting, the result can be shown for direct products of torsion-free hyperbolic groups by analysing the structure of centralisers. A similar analysis carries over to the pro-$p$ setting. Furthermore, the proof above, which relies on \Cref{posgrad}, is more flexible and suitable in our context, and will enable us to understand more than just automorphisms of direct products, but also surjections between them.

Combining similar arguments as in the proof of \Cref{s1}, we will see in \Cref{s2} how one can ensure scenarios when direct product decompositions can be read off from the pro-$p$ completion. However, we will require first two lemmas. 

\begin{lemma} \label{s3}  Let $n\geq 1$ be an integer and let $K_i$, with $1\leq i\leq n$, be finitely generated residually-$p$ groups with positive $p$-gradient. Consider the direct product $K=K_1\times \cdots \times K_n$ and suppose that $K'$ is an intermediate group $K\leq K'\leq K_{\hp}$ such that the index $|K': K|$ is finite. Then $K=K'$. 
\end{lemma}
\begin{proof}  We proceed by induction. Let $n=1$. Since $K'$ is dense in $K_{\hp}$, $\RG(K')\geq \RG(K)$. However, $\RG(K)=|K': K| \cdot \RG(K')$ and $\RG(K)>0$, so $|K':K|=1$ and the conclusion follows. When $n\geq 2$, consider the direct product decomposition $K_{\hp}\cong (K_1)_{\hp}\times \cdots \times (K_n)_{\hp}$. Denote by $\pi_1\po K_{\hp}\lrar(K_1)_{\hp} $ and $\pi^1\po  K_{\hp}\lrar(K_2)_{\hp}\times \cdots \times (K_n)_{\hp}$ the canonical projections. By considering the inclusion $K_1\leq \pi_1(K'\cap (K_1)_{\hp})\leq (K_1)_{\hp}$ and applying the lemma for $n=1$, we deduce that $K'\cap (K_1)_{\hp}=K_1$. Lastly, we can apply the lemma to the chain $K_2\times \cdots \times K_n\cong K/K_1\leq K'/K_1\leq K_{\hp}/(K_1)_{\hp}\cong (K/K_1)_{\hp}$, which, by the induction hypothesis, tells us that $K/K_1=K'/K_1$. Thus $K=K'$.
\end{proof}
Finally, we can state a sufficient condition to ensure that direct factors are recognised by the pro-$p$ completion. 
\begin{lem} \label{s4} Let $n\geq 1$ be an integer and let $H_i,$ with $1\leq i\leq n$, be finitely generated residually-$p$ groups with positive $p$-gradient. Let $H$ be a finitely generated group with $H_{\hp}\cong (H_1)_{\hp}\times \cdots \times (H_n)_{\hp}$ such that the subgroup generated by all the intersections $H\cap (H_i)_{\hp}$ has finite index in $H.$ Then $H=(H\cap (H_1)_{\hp})\times \cdots \times (H\cap (H_n)_{\hp})$.
\end{lem}
\begin{proof} We show by induction that, given the assumptions of the lemma, $H\cap (H_i)_{\hp} $  is dense in $(H_i)_{\hp}$ for all $i$ and that $H=(H\cap (H_1)_{\hp})\times \cdots \times (H\cap (H_n)_{\hp})$. The case $n=1$ is trivial. Now we consider $n\geq 2$ and suppose that it holds for $n-1.$ By assumption, the subgroup $\lan H\cap (H_1)_{\hp}, \dots, H\cap (H_n)_{\hp}\ran$ has finite index and hence $\lan \bar{H\cap (H_1)_{\hp}}, \dots, \bar{H\cap (H_n)_{\hp}}\ran$ is open in $H_{\hp}$, so each $\bar{H\cap (H_i)_{\hp}}$ is open in $(H_i)_{\hp}.$  We consider the group $H^1=H/H\cap (H_1)_{\hp}$. The natural map \[H^1\lrar \left((H_1)_{\hp}/\bar{H\cap (H_1)_{\hp}}\right)\times (H_2)_{\hp}\times \cdots \times (H_n)_{\hp}\] is a pro-$p$ completion map and is injective. By assumption, the first factor $(H_1)_{\hp}/\bar{H\cap (H_1)_{\hp}}$ is a finite $p$-group and has trivial intersection with $H^1$. In particular, $\UU=(H_2)_{\hp}\times \cdots \times (H_n)_{\hp}$ is an open subgroup of $(H^1)_{\hp}$. Suppose that $(H_1)_{\hp}/\bar{H\cap (H_1)_{\hp}}$ has size $k\geq 1$ and take the finite-index subgroup $H_0^1\leq H^1$ defined by $H_0^1=H^1\cap \UU $. Denote by $\pi^1\po (H^1)_{\hp}\lrar (H_2)_{\hp}\times \cdots \times (H_n)_{\hp}$ the canonical projection.  We know that $|H^1\po H_0^1|=k$ and that the two groups project injectively to  $(H_2)_{\hp}\times \cdots \times (H_n)_{\hp}\cong (H_0^1)_{\hp}$. By  the induction hypothesis, we also know that $H_0^1=K_2\times \cdots \times K_n$ with $(K_i)_{\hp}\cong (H_i)_{\hp}$ (where $K_i=H_0^1\cap (H_i)_{\hp}$). So, by \Cref{s3}, $k=1$, implying that $H^1=H_0^1=(H\cap (H_2)_{\hp})\times \cdots \times (H\cap (H_n)_{\hp})$ and that $H\cap (H_1)_{\hp}$ is dense in $(H_1)_{\hp}$. The conclusion follows.
\end{proof}
We finish this subsection with the following proposition which enables to recover direct product decompositions from the pro-$p$ completion; which think that this proposition may be of independent interest. It is also an important ingredient in the proof of \Cref{productsintro}.

\begin{prop} \label{s2}  Let $n\geq 1$ be an integer and let $H_i$, with $1\leq i\leq n$, be finitely generated residually-$p$ groups with positive $p$-gradient. Suppose that $H$ is a torsion-free and residually-$p$ finitely generated group. Suppose that  $H_{\hp}\cong (H_1)_{\hp}\times \cdots \times (H_n)_{\hp}$ and that there are pairwise-commuting infinite subgroups $N_i\leq H$ such that the natural map $N_1\times \cdots \times N_n \lrar \lan N_1, \dots, N_n\ran$ is an isomorphism and the subgroup $\lan N_1, \dots, N_n\ran $ has finite index in $H$.  Then there exist subgroups $N_i'\leq H$, each commensurable to $N_i$ respectively, such that $H= N_1'\times \cdots \times N_n'$.
\end{prop}

\begin{proof} Replacing each $N_i$ by one of its finite-index subgroups, we can assume that each $N_i$ is normal in $H.$ For each $1\leq j\leq n$, denote by $\pi_j \po H_{\hp}\lrar (H_j)_{\hp}$ the canonical projection. Since each $N_i$ is infinite, then each closure $\bar{N_i}$ is an infinite normal  subgroup of $H_{\hp}$. Hence, for all $i,$ there exists $j$ such that $\pi_j(\bar{N_i})$ is an infinite normal subgroup of $(H_j)_{\hp}$, which, by \Cref{posgrad}, must have finite index. Proceeding exactly as in the proof of \Cref{s1}, we can show that, for every $i$, the choice of such $j$ (denoted by $j(i)$) is unique. For every $i$ and $j\neq j(i)$, the subgroup $\pi_j(\bar{N_i})\leq (H_j)_{\hp}$ is finite. So it is clear that, replacing again each $N_i$ by smaller finite-index subgroups, we can suppose that for every $i$ and $j\neq j(i)$, $\pi_j(\bar{N_i})=1$ and hence $N_i\leq H\cap (H_{j(i)})_{\hp}$. Since the subgroup generated by all the $N_i$ is still finite-index in $H$, then we can apply \Cref{s4} to derive that $H= N_1'\times \cdots \times N_n'$, where $N_i'=H\cap (H_{j(i)})_{\hp}.$
\end{proof}

\subsection{Profinite completions of limit groups}
Here we recall a few results about profinite completions of limit groups that will be useful later. Zalesskii-Zapata use Sela's hierarchy of limit groups to show that if $L$ is a non-abelian limit group then $\hat L$ acts faithfully and irreducibly on a profinite tree. They use the description of this action to establish a number of properties that were known for abstract limit groups regarding their cohomological dimension and centralisers of elements, such as the following one. 

\begin{prop}[Corollary 4.4, \cite{Zal20}]\label{centre} Let $L$ be a non-abelian limit group. Then $\hat L$ is centreless.
\end{prop}
The following consequence of \Cref{centre} will also be important in the proof of \Cref{productsintro}.
\begin{prop} \label{two} Let $ L$ be a non-abelian limit group. Suppose that there are two closed subgroups $\AA$ and $\BB$ of $\hat L$ with the following two properties:
\begin{itemize}
    \item For all $a\in \AA$ and $b\in \BB$, the elements $a$ and $b$ commute.
    \item For some prime $p$, there are embeddings $\Z_p\lc \AA$ and $\Z_p\lc \BB$.
\end{itemize}
Then the closed subgroup generated by $\AA$ and $\BB$ in $\hat L$ is not open.    
\end{prop}

\begin{proof} We proceed by contradiction. Suppose that the closed subgroup $ \lan \AA, \BB\ran$ is open. Then it is isomorphic to $\hat L_1$ for some finite-index $L_1\lo L$, which will again be a non-abelian limit group. Hence we can suppose, without loss of generality, that $\lan \AA, \BB\ran=\hat L$. Since $\hat L$ is centreless, by \Cref{centre}, it is clear that $\AA\cap \BB=\{1\}$. So $\hat L\cong \AA\times \BB$ and, by  \cite[Proposition 4.2.4]{Rib17},   $\hat L$ is projective. However, by assumption, $\Z_p^2\lc \hat L$, which contradicts the projectivity of $\hat L$. 
\end{proof}

\subsection{Proof of \Cref{productsintro}}
We reformulate \Cref{productsintro} in an equivalent way to simplify the notation of the proof. For doing so, we simply notice that any product of free or surface groups has the form $\Ga=S_1\times \cdots \times S_n\times \Z^k$ with $n\geq 1$, $k\geq 0$ and all of the $S_i$ being non-abelian free or hyperbolic surface groups. 
\begin{thm}[\Cref{productsintro}] \label{products2} Let $n\geq 1$ and $k\geq 0$ be integers and let $G$ be a finitely presented residually free group. Let $S_1, \dots, S_n$ be non-abelian free or surface groups and consider the direct product $\Ga=S_1\times  \cdots \times S_n\times \Z^k$. If $\hat G\cong \hat \Ga$, then $G\cong \Ga$.
\end{thm}

We will quickly summarise the different stages of the proof of \Cref{products2}.

\begin{itemize}
    \item[(Step 1)] We consider $G/Z(G)$ to work under the additional hypothesis of $G$ having trivial centre. This case will be enough as we explain in \Cref{c7}.
    \item[(Step 2)] We view $G$ as the subdirect product of $n$ non-abelian limit groups $L_1\times \cdots \times L_n$ with claims \ref{c1} and \ref{c2}. 
    \item[(Step 3)] Using the structure theory of finitely presented subgroups of direct products of limit groups due to Bridson--Howie--Miller--Short (\Cref{brihow}), in combination with properties of Hirsch lengths in nilpotent groups (\Cref{NilpotentSection}), we see in \Cref{c4} that, in fact, $G$ has  finite index in $L_1\times \cdots \times L_n$. 
    \item[(Step 4)] This clearly shows that $G$ has a finite-index subgroup $H$ that is a direct product of $n$ limit groups. Since each $L_i$ has positive $p$-gradient, we can apply the methods of  \Cref{GradientSection}  to lift the  direct product decomposition structure of $H$ to $G$. This is done in \Cref{c5}.
    \item[(Step 5)] Now we have an isomorphism $G\cong G_1\times \cdots \times G_n$, with limit groups $G_i$. We show in \Cref{c6} that each is a free or a surface group using, again, the results from \Cref{GradientSection}.
\end{itemize}

After having given this sketch,  we move on to explain the proof of \Cref{products2} in detail. As anticipated above, the following lemma will allow us to work under the extra assumption of a centreless group $G$.
\begin{lem}[\cite{Bau67a}, Lemma 4] \label{nocentre} Let $G$ be a residually free group. Then $G/Z(G)$ is a residually free group with trivial centre. 
\end{lem}

We will consider $G_0=G/Z(G)$, whose profinite completion is $\hat{G_0}\cong\hat{G}/\bar{Z(G)}$. Since each $\hat S_i$ is centreless, we know that $Z(\hat G)=\hat \Z^k$ and hence $\hat{G_0}\cong  \hat S_1\times \cdots \times \hat S_n\times \AA,$ where $\AA=Z(\hat G)/\bar{Z(G)}$ is an abelian profinite group. 
\begin{claim} \label{c1} The group $G_0$ is finitely presented and residually free, and its profinite completion is isomorphic to $\hat {G_0}\cong \hat S_1\times \cdots \times \hat S_n$.
\end{claim}
\begin{proof}
By \Cref{nocentre} and the finite generation of $Z(G)$, we know that $G_0$ is still a finitely presented residually free group. By \Cref{subprod}, $G_0$ is a subdirect product $G_0\hookrightarrow L_1\times \cdots \times L_m$ of $m$ limit groups, for some $m$.
We pick such $m$ to be minimal. Hence, in particular, $G_0$ intersects non-trivially each factor $L_i$. Since $G_0$ is centreless, each $L_i$ is non-abelian. So by  \Cref{centre}, $\hat L_i$ is centreless, too. 

For each $1\leq j\leq m$, denote by $\pi_j\po   L_1\times \cdots \times L_m \lrar L_j$ the projection onto the $j$-th coordinate. By \Cref{sep}, the induced map in profinite completions $ \hat{G_0}\hookrightarrow \hat{L_1}\times \cdots \times \hat{L_m}$ is injective. Since we also know that  the maps $\hat{\pi_j}\po \hat {G_0} \lrar \hat L_j$ are surjective and that each $\hat L_j$ is centreless, it follows that $\hat{\pi_j}(\AA)=1$ for all $j$, implying that $\AA=1$.
\end{proof}

\begin{claim} \label{c2} We must have $n=m$. 
\end{claim}
\begin{proof}
Since limit groups are torsion-free and $G\cap L_j\neq 1$ for each $1\leq j\leq m$, we must have an injection $\Z^m\hookrightarrow G_0$. By \Cref{sep}, this leads to another injection in profinite completions $\hat \Z^m\hookrightarrow \hat {G_0}\cong \hat S_1\times \cdots \times \hat S_n$, which, in combination with the standard fact that no $\hat S_i$ contains $\Z_p^2$, implies that $m\leq n$. 

Lastly, we want to show that $n\leq m$. We proceed by contradiction. Assume that $n\geq m+1$. For each $1\leq i\leq n$, consider a closed infinite pro-$p$ cyclic subgroup $\Z_p\cong \ZZ_i\leq \hat S_i$. Since $n\geq m+1$, there must exist two different $\ZZ_{k_1}$ and $\ZZ_{k_2}$ such that both of their images $\hat{\pi_j}(\ZZ_{k_1})$ and $\hat{\pi_j}(\ZZ_{k_2})$ in some $ \hat L_j$ are isomorphic to $\Z_p$. However, this would contradict  \Cref{two}, since then the two pairwise commuting closed groups $\hat{\pi_j}(\hat S_{k_1})$ and $\hat{\pi_j}(\prod_{i\neq k_1} \hat S_{i})$ both would contain a copy of $\Z_p$ and would generate $\hat L_j$. This contradiction shows that $n\leq m$, completing the proof of \Cref{c2}.\end{proof}

We will now introduce additional notation required for the proof of \Cref{c4}, which states that the injection $G_0\hookrightarrow L_1\times \cdots \times L_m$ has a finite-index image. By \Cref{brihow},  there exists an integer $N$, a finite-index subgroup $E\leq L_1\times \cdots \times L_m$ and a finite-index subgroup $G_1\leq G_0$ such that $\ga_N E\leq G_1$. We notice that the subgroup $E_1=(E\cap L_1)\times \cdots \times (E\cap L_n)\leq E$ has finite index in $E$ and that, by defining $G_2=E_1\cap G_1$, we still have the property that $G_2\leq G$ has finite index and that $\ga_N E_1\leq E_1\cap G_1=G_2$. We denote the limit groups $U_i=E\cap L_i$. The projection of the direct product $U_1\times  \cdots \times U_n$ onto its $i$-th coordinate  will still be denoted by $\pi_i$. The closure of $G_2$ inside $\hat G_0$ is an open subgroup $\bar{G_2}\lo \hat G_0$. Denote $\KK_i=\bar{G_2}\cap \hat S_i$ and notice that $\KK=\KK_1\times \cdots \times \KK_n$ is an open subgroup of $\bar{G_2}.$  We define $G_3=\KK\cap G_0$, which has the property that  $\hat G_3\cong \KK$.   Summarising, $G_3\leq G_0$ has finite index,  $\hat G_3\cong \KK_1\times \cdots \times \KK_n$, $E_1=U_1\times \cdots \times U_n$, and there is an injection $f\po G_3\hookrightarrow E_1$ such that each $\pi_j(G_3)$ has finite index in $U_j$. Since each $\KK_i$ is open in $\hat S_i$, they are profinite completions of non-abelian free or hyperbolic surface groups. We denote by $(\KK_i)_{p}$ the maximal pro-$p$ quotient of the profinite group $\KK_i$. It is not hard to see that $(\KK_i)_{p}$ is the pro-$p$ completion of a non-abelian free or hyperbolic surface group and that $(G_3)_{\hp}\cong (\KK_1)_{p}\times \cdots \times (\KK_n)_{p}.$

\begin{claim} \label{c3} The induced map in pro-$p$ completions $f_{\hp}\po (G_3)_{\hp}\lrar (E_1)_{\hp}$ has a finite index image. 
\end{claim}
\begin{proof} From the proof of \Cref{c2} we know that there exists an injection $\Z^n\hookrightarrow G_0$, hence there will also exist an injection $\Z^n\hookrightarrow G_3$.  By \cite[Theorem 7.12]{Mor22}, we know that the induced maps in pro-$p$ completions $\Z_p^n\hookrightarrow (G_3)_{\hp}$ and $\Z_p^n\hookrightarrow (E_1)_{\hp}$ are injective. Since each $(\KK_i)_p$ is the pro-$p$ completion of a non-abelian free or hyperbolic surface group, no $(\KK_i)_p$ contains a copy of $\Z_p^2$. This implies that each intersection $\Z_p^n\cap (\KK_i)_p$ is infinite. In particular, since $f_{\hp}$ is injective in $\Z_p^n$, the latter implies that the restriction of $f_{\hp}$ to each $i$-th coordinate $(\KK_i)_p$ of $(G_3)_{\hp}$ has an infinite image. Then, for each $1 \leq i\leq n$, there exists a subscript $j=j(i)$ such that $(\pi_j\circ f)_{\hp}((\KK_i)_p)\leq (U_j)_{\hp}$ is infinite. For each $i$, there may be a priori several $j$ with the previous property, but let us consider any such choice $j(i)$ for each $i$. We observe that, since $H_i$ is normal in $G_3$, $(\pi_{j(i)}\circ f)_{\hp}((\KK_i)_p)$ is normal in the finite-index subgroup $(\pi_{j(i)}\circ f)_{\hp}((G_3)_{\hp})$ of $(U_{j(i)})_{\hp}$. By  \Cref{posgrad}, this implies that $(\pi_{j(i)}\circ f)_{\hp}((\KK_i)_p)$ has finite index in $(U_{j(i)})_{\hp}.$ With this observation in mind, we notice that for different $i_1\neq i_2$, we must have $j(i_1)\neq j(i_2)$. The reason is that if we had $j=j(i_1)=j(i_2)$ for $i_1\neq i_2$, then the open commuting subgroups $(\pi_{j}\circ f)_{\hp}((\KK_{i_1})_p)$ and $(\pi_{j}\circ f)_{\hp}((\KK_{i_2})_{p})$ of $(U_{j})_{\hp}$ intersect in an open abelian subgroup, contradicting the fact that $(U_j)_{\hp}$ is not virtually abelian (because $U_j$ is not virtually abelian). Thus, for each $i$ there is a unique choice of $j(i)$, and we have a bijection $j\po \{1, \dots, n\}\lrar \{1, \dots, n\}.$ Furthermore, for each $i$ and $j\neq j(i)$, the image   $(\pi_{j}\circ f)_{\hp}((\KK_{i})_{p})$  in $(U_{j})_{\hp}$ is finite. From this, it is easy to see that  $f_{\hp}((\KK_{i})_{p})\cap (U_{j(i)})_{\hp}$ is open in $(U_{j(i)})_{\hp}$ and thus $f_{\hp}$ has a finite-index image. 
\end{proof}

\begin{claim} \label{c4} The injection $G_0\hookrightarrow L_1\times \cdots \times L_m$ has a finite-index image.
\end{claim}
\begin{proof} Since $G_3\leq G_2\leq G_0$ and $E_1\leq E$ are finite-index subgroups, it suffices to check that the injection $G_2\hookrightarrow E_1$ has a finite-index image. We know that $\ga_N E_1\leq G_2$, so we can consider the induced injection of finitely generated nilpotent groups $h\po G_2/\ga_N E_1\hookrightarrow E_1/\ga_N E_1.$ We know from \Cref{c3} that $(G_3)_{\hp}\lrar (E_1)_{\hp}$ has a finite-index image, and so does $h_{\hp}$. Thus, by \Cref{propimage}, the map $h$ has a finite-index image, as we wanted.
\end{proof}

\begin{claim}  \label{c5} There exist subgroups $G_1, \dots, G_n\leq G_0$ with positive $p$-gradient such that $G_0=G_1\times \cdots \times G_n.$
\end{claim}
\begin{proof} By \Cref{c4}, the subgroup $(G_0\cap L_1)\times \cdots \times (G_0\cap L_m)\leq G_0$ has finite index and  each $G_0\cap L_i$ is a non-abelian limit group, implying that they have positive $p$-gradient. The conclusion follows directly from \Cref{s2}.
\end{proof}

\begin{claim} \label{c6} There is an isomorphism $G_0\cong S_1\times\cdots \times S_n$. 
\end{claim}
\begin{proof} By Claims \ref{c1} and \ref{c5},  $\hat G_0\cong \hat S_1\times \cdots\times \hat S_n$ and, for some   subgroups $G_i$ with positive $p$-gradient, $G_0=G_1\times \cdots \times G_n$. We know that free and surface groups are distinguished from each other by their pro-$p$ completion. Combining this observation with \Cref{s1}, it is enough to show that each $G_i$   is either a free or a surface group, which we will verify only for $i=1$ (for simplicity). By \Cref{Melnip}, it is enough to show that for every finite-index subgroup $K_1\leq G_1$, the pro-$p$ completion $(K_1)_{\hp}$ is the pro-$p$ completion of either a free or a surface group. To prove this, let us fix a finite-index subgroup $K_1$ of $G_1$. The subgroup $K_1\times G_2\times \cdots \times G_n\leq G$ has finite index in $G$ and so it has the same profinite completion as a finite-index subgroup $H$ of $S_1\times \cdots \times S_n$. We apply \Cref{s2} to $H$ and its subgroups $N_i=H\cap S_i$, since the injection $(H\cap S_1)\times \cdots \times (H\cap S_n)\leq H$ must have finite index. This way, we deduce that there are subgroups $N_i'$ of $H$, each commensurable to $H\cap S_i$ respectively, such that  $H\cong N_1'\times \cdots \times N_n'$. From this, we get the isomorphisms $(K_1)_{\hp}\times (G_2)_{\hp} \times \cdots \times (G_n)_{\hp}\cong H_{\hp}\cong (N_1')_{\hp}\times \cdots \times (N_n')_{\hp}$. Notice that each torsion-free group $N_i'$ is commensurable to either a free or a surface group and so, by Nielsen realisation, each $N_i'$ is itself either a free or a surface group.  By \Cref{s1}, there exists $1\leq i\leq n$ such that $ (K_1)_{\hp}\cong (N_i')_{\hp}$, as we wanted.
\end{proof}
\begin{claim} \label{c7} There is an isomorphism $G\cong \Ga$.
\end{claim}
\begin{proof} There is a short exact sequence 
\begin{equation} \label{e8} 1\lrar \Z^k\lrar G\lrar G_0\lrar 1\end{equation}
whose induced sequence in profinite completions 
\begin{equation*}   1\lrar {\hat \Z}^k\lrar \hat G\lrar \hat G_0\lrar 1\end{equation*}
is exact and splits, by construction. 
By \cite[Lemma 8.3]{Wilton17} and the goodness of $G_0$, this means that the short exact sequence of (\ref{e8}) splits. So $G\cong \Z^k\times G_0$ and the conclusion follows from \Cref{c6}.
\end{proof}

\section{Other calculations of virtual Betti numbers and further questions} \label{questions}
\subsection{Divergence between first usual and $L^2$-Betti numbers}
 
\begin{thm} \label{boundedvd} Let $G$ be a limit group with $\vd(G)<\infty$. Then $G$ is either free, free abelian or surface group.
\end{thm}
Before proving \Cref{boundedvd}, we state a proposition which will greatly simplify the proof.

\begin{prop} \label{BDprop} Suppose that a finitely generated and subgroup separable group $
G$ splits as a finite and connected graph of groups. If $G$ is residually-(amenable and locally indicable) and there exists a vertex group $G_v$ such that $\vd(G_v)=\infty$, then $\vd(G)=\infty$.
\end{prop}
\begin{proof} For all $k\geq 1$, there exists a finite-index subgroup $H_v\leq G_v$ such that $\d(H_v)\geq k$. By subgroup separability, there is a finite-index subgroup $H\leq G$ such that $H\cap G_v=H_v$ and hence the corresponding graph of groups decomposition for $H$ has cyclic edges and one of its vertex isomorphic to $H_v$. By \Cref{L^2general},  $\d(H)\geq \d(H_v)\geq k$. So $\vd(G)\geq k$ for all $k$ and the conclusion follows.
\end{proof}
We are finally ready to prove \Cref{boundedvd}.

\begin{proof}[Proof of \Cref{boundedvd}] If $G$ has cohomological dimension at most two, then, by \Cref{limb}, $\vb_2(G)+1=\vd(G)$. Hence $\vb_2(G)<\infty$ and the conclusion follows from \Cref{mainlim}. It suffices to deal with the case where $G$ has cohomological dimension at least $3$. Sela \cite{sela1} proved that limit groups admit a finite cyclic hierarchy terminating in free abelian groups (which are the only rigid limit groups) (see Theorem \ref{hierarchy_thm}). Since $\cd(G)\geq 3$, one of the rigid groups in which the hierarchy of $G$ terminates must be a free abelian group $\Z^k$ with $k\geq 3$ (let us denote it by $A$).  Finitely many applications of \Cref{BDprop} tell us that, in order to conclude that $\vd(G)=\infty$, it is enough to show that at some level of the finite hierarchy of $G$ lies a group $H$ with $\vd(H)=\infty$. If $G$ is not abelian, we claim that such $H$ can be taken to be the group which lies above the aforementioned $A$ in the hierarchy of $G$. $H$ is non-abelian and hence splits as a graph of groups with cyclic (possibly trivial) edges, with at least two vertices. Furthermore, one of the vertices $H_v$ is isomorphic to $A\cong \Z^k$ (with $k\geq 3$) and, by construction, any of the neighbouring vertex groups $H_u$ of $H_v$ is not virtually cyclic. 

Let $\ell \geq 1$ be any integer. By Corollary \ref{star-tree}, there is a finite-index subgroup $\ti{H}$ of $H$ that admits a splitting as a graph of groups with cyclic (possibly trivial) edges with the following two properties: 
\begin{enumerate}
    \item The underlying graph contains a subtree $T$ with either $\ell$ or $\ell+1$ vertices.
    \item There are at least $\ell$ vertices  $v\in V(T)$ such that the corresponding vertex group $\ti{H}_v$ is isomorphic to a finite-index subgroup of $A$ (so $\ti{H}_v\cong A\cong \Z^k$ and $\d(\ti{H}_v)\geq 3$).
\end{enumerate}
By \Cref{L^2general}, \[\d(\ti{H})\geq \d(\pi_1(\ti{H}, T)) \geq \sum_{v\in T} \d(\ti{H}_v)-(|V(T)|-1)\geq 3\ell-\ell=2\ell. \]
This shows that $\vd(H)=\infty$, completing the proof.
\end{proof}
\subsection{Virtual homology of RAAGs and manifolds} \label{virtual_Bettis}
We compute the virtual homology of other classes of groups, including Right-Angled Artin groups and fundamental groups of closed, hyperbolic $3$ manifolds. 

\begin{prop} \label{RAAGvb} Let $G$ be a RAAG, let $k$ be a field and let $n\in \mathbb{N}$. Then $\vb_n(G)<\infty$ if and only if either $\cd(G)<n$ or $G$ is free abelian.
\end{prop}

\begin{proof} 
We only sketch the proof of \Cref{RAAGvb}, which is very simple. If $\cd(G)\geq n$ and $G$ is not free abelian, then the defining graph $\Ga$ of $G$ contains a subgraph of $n+1$ vertices which is not complete but contains a complete graph of $n$ vertices. This directly implies that $G$ retracts onto a subgroup isomorphic to $\Z^n*_{\Z^m}$, where $n>m$. Now \Cref{cyclic} applies and we conclude that $\vb_n^k(G)=\infty$.
\end{proof}

\begin{rmk} \label{agol_reid_rmk} After Agol's resolution of the virtual Haken conjecture \cite{Agol13}, we know that fundamental groups of closed hyperbolic 3-manifolds are large. This implies that $\vb_1(M)=\infty$ and, by Poincaré duality, that $\vb_2(M)=\infty$. Regarding higher dimensions,
Alan Reid pointed out that the fundamental group of an arithmetic hyperbolic $n$-manifold $M$ of simplest type has $\vb_i^\mathbb{Z}(\pi_1(M))=\infty$ for $1\le i <n$. This can be proven by induction as follows: as in the lines above, the assertion holds for $n \le 3$. For $n\ge 4$, we use the fact that $M$ has immersed totally geodesic submanifold $M'$ of dimension $n-1$. By the induction hypothesis, $\vb_i^\mathbb{Z}(\pi_1(M'))=\infty$ for $1\le i <n-1$. All immersed totally geodesic submanifolds contribute to homology in finite covers by \cite{submfld}, which implies that  $\vb_i^\mathbb{Z}(\pi_1(M))=\infty$ for $1\le i < n-1$. Finally, by Poincar\'e duality, we also have that $\vb_{n-1}^\mathbb{Z}(\pi_1(M))=\infty$.
\end{rmk}

\subsection{Further questions and virtual invariants}
\Cref{mainlim} shows that when an HGFC-group $G$ has $\vb_2^k(G)<\infty$, then $G$ is either a free or a surface group. This suggests the following.

\begin{qu} \label{questionvb} Let $G$ be a hyperbolic group of cohomological dimension two. If $G$ has a finite virtual second Betti number, does it follow that $G$ is a surface group?
\end{qu}

%Recall that the proof of \Cref{pversion} crucially relied on $\vd(G)$ being finite. The number $\vd(G)$ was defined in \Cref{vd} as  \[\sup\{ \d(H) \, \vert \, H \text{ is a finite-index subgroup of } G\}, \]where $\d(H)=b_1(H)-\b(H)$. For finitely presented groups $G$ with $\cd(G)\leq 2$, bounded $\vb_2(G)$ implies bounded $\vd(G)$ (in  \Cref{limb} it is shown that $\vd(G)\leq \vb_2(G)+1$). The same argument shows that, if we additionally suppose $b_2^{(2)}(G)=0$, then $\vd(G)=\vb_2(G)+1$. This is true for one-relator groups and for the cyclic hierarchies that we are considering \cite{l2_lim}. Nevertheless,  the question of which groups of cohomological dimension two have $b_2^{(2)}(G)=0$ remains a mystery. In this regard, we are grateful to Sam Fisher for pointing out a conjecture of D. Wise \cite{Wise22} that expects the vanishing of $b_2^{(2)}(G)$ to be equivalent to the coherence of $G$. In higher cohomological dimensions, there is no obvious relation between $\vd(G)$ and $\vb_2(G)$, so we also pose the following question about the boundedness of $\vd(G)$.
We remark that this question is closely related to two of the most notorious open questions revolving around hyperbolic groups (namely, whether all hyperbolic groups are residually finite, and whether every one-ended hyperbolic group contains a surface subgroup). \par \smallskip
A deeper understanding of where the second cohomology of $G$ comes from (in terms of Sela's finite cyclic hierarchy) was crucial in the proof of \Cref{boundedvd}, which raises the following natural question.
\begin{qu}\label{questionvb2}  Let $G$ be a torsion-free, finitely generated and residually finite group such that $\b(G)>0$ and  \[\sup\left\{ b_1(H)-\b(H) \, \vert \, H \text{ is a finite-index subgroup of } G\right\}<\infty.\]
Does it follow that $G$ is either a non-abelian free group or a hyperbolic surface group? 
\end{qu}

%The third property that we introduced in \Cref{ABC} is property $(\Tfab)$, which consists on the torsion-freeness of the abelianisation of all finite-index subgroups. Our following question is about $(\Tfab)$ and is supported by computational evidence obtained by the authors. It was also asked by Ilir Snopce and Pavel Zalesskii in discussions that happened during the workshop "New Trends around Profinite Groups 2022", held in Varese.

%\begin{qu} \label{questionB} Let $G$ be a one-ended limit group with the property that all of its finite-index subgroups $H<G$ have a torsion-free abelianisation. Does it follow that $G$ is isomorphic to either a free abelian or a surface group?
%\end{qu}

At this point we have seen many examples of groups, some of which admitting interesting hierarchies, where the boundedness of $\vb_2$ (or higher virtual homology notions) is very exceptional. In particular,  most limit groups (with the only exception of free, free abelian and surface groups) have unbounded $\vb_2$. We would therefore like to introduce another notion which might aid in obtaining a finer classification of limit groups based on (virtual) homological invariants.
	\begin{defn}
		Let $G$ be a finitely generated group. Define the \emph{virtual $i$-th cohomology spectrum} of $G$, with coefficients in a field $k$, to be the set 
    \begin{equation*}
        \mathrm{Spectrum}_{i}^k(G)=\{\dim_k H^i(H;k)\vert H\le G \text{ of finite index}\}.
    \end{equation*}
	\end{defn}
	Cohomological goodness implies that if $G$ is a limit group or an HGFC-group,  $$\mathrm{Spectrum}_{i}^{\mathbb{F}_p}(\hat{G})=\mathrm{Spectrum}_{i}^{\mathbb{F}_p}(G).$$
    In other words, the virtual $i$-th cohomology spectrum with coefficients in $\mathbb{F}_p$ is a profinite invariant of $G$ within good groups. One can strengthen the notion of a virtual cohomology spectrum, and define the \emph{filtered virtual $i$-th cohomology spectrum} of $G$, with coefficients in a field $k$, to be
	\begin{equation*}
		\mathcal{F}{\mathrm{Spectrum}_{i}}^k(G)=\bigg\{\frac{\dim_k H^i(H;k)}{[G:H]}\bigg\vert H\le G \text{ of finite index}\bigg\}.
	\end{equation*}
 This refined version also keeps track of the depth of the finite-sheeted coverings of $G$ in which large $i$-th homology is exhibited.
	Again,  the filtered virtual $i$-th cohomology spectrum of $G$ (with coefficients in $\mathbb{F}_p$) is a profinite invariant of good groups. This raises the following questions:
	\begin{qu} \label{spectrumq1}
    Let $G$ be a limit group or an HGFC-group.
		\begin{enumerate}
		    \item Which numbers appear in the virtual $i$-th cohomology spectrum of $G$? 
                \item Which limit groups and HGFC-groups are characterized by their filtered second Betti spectrums?
		\end{enumerate}
	\end{qu}
	The proof of Theorem \ref{mainlim} relies heavily on virtual retractions (see Section \ref{sep}). Some information is lost in the process of transferring homological data via virtual retractions. This is why the 
 proof of \Cref{boundedvd} required an additional argument. In particular, answering Question \ref{spectrumq1} should involve methods different to the ones used in this paper.

%\renewcommand{\refname}{Bibliography}
%\begin{thebibliography}{}
%	\input{virtual_b_2.bbl}
%\end{thebibliography}
%\bibliography{biblio}
\bibliographystyle{plain}

\vspace{10mm}

\textsc{Mathematisches Institut, Universität Bonn, Endenicher Allee 60, 53115 Bonn, Germany} \\
\textit{E-mail address:} \href{mailto:fruchter@math.uni-bonn.de}{\texttt{fruchter@math.uni-bonn.de}}
\medskip

\textsc{Mathematical Institute, University of Oxford, Radcliffe Observatory, Andrew Wiles Building, Woodstock Rd, Oxford OX2 6GG} \\
\textit{E-mail address:} \href{mailto:morales@maths.ox.ac.uk}{\texttt{morales@maths.ox.ac.uk}}
\end{document}